\documentclass[reqno,11pt]{amsart}
\usepackage{amssymb}
\usepackage{verbatim}
\newif\ifpdf
\ifpdf
  \usepackage[pdftex]{graphicx}
  \usepackage[pdftex]{hyperref}
\else
  \usepackage{graphicx}
\fi
\numberwithin{equation}{section}       
 \theoremstyle{plain}    
 \newtheorem{thm}{Theorem}[section]
 \numberwithin{equation}{section} 
 \numberwithin{figure}{section} 
 \theoremstyle{plain}
 \theoremstyle{plain}    
 \newtheorem{cor}[thm]{Corollary} 
 \theoremstyle{plain}    
 \newtheorem{prop}[thm]{Proposition} 
 \theoremstyle{plain}    
 \newtheorem{lem}[thm]{Lemma} 
 \theoremstyle{remark}
 \newtheorem{rem}[thm]{Remark}
 \theoremstyle{definition}

\newtheorem*{thmA}{Theorem A} 
 
\newtheorem*{thmB}{Theorem B} 
\newtheorem*{thmC}{Theorem C} 
\newtheorem*{thmD}{Theorem D}
\newtheorem*{thmE}{Theorem E}

\theoremstyle{definition}
\newtheorem{defi}[thm]{Definition}

\newtheorem*{ackn}{Acknowledgements}
\newtheorem*{org}{Organization of the article}

\newcommand{\C}{{\mathbb{C}}}

\newcommand{\PP}{{\mathbb{P}}}

\newcommand{\R}{{\mathbb{R}}}

\newcommand{\cC}{{\mathcal{C}}}

\newcommand{\cE}{{\mathcal{E}}}
\newcommand{\cF}{{\mathcal{F}}}
\newcommand{\cH}{{\mathcal{H}}}

\newcommand{\cO}{{\mathcal{O}}}

\newcommand{\cT}{{\mathcal{T}}}

\newcommand{\cM}{{\mathcal{M}}}

\newcommand{\cP}{{\mathrm{PSH}}}

\newcommand{\cK}{{\mathcal{K}}}

\renewcommand{\a}{\alpha}
\renewcommand{\b}{\beta}

\newcommand{\de}{\delta}
\newcommand{\e}{\varepsilon}
\newcommand{\f}{\varphi}

\newcommand{\p}{\psi}

\newcommand{\FS}{\mathrm{f}}
\newcommand{\MA}{\mathrm{MA}}
\newcommand{\Amp}{\mathrm{Amp}\,}
\newcommand{\Hilb}{\mathrm{h}}

\newcommand{\vol}{\operatorname{vol}}

\newcommand{\conv}{\operatorname{Conv}}

\newcommand{\ca}{\operatorname{Cap}}

\newcommand{\eq}{{\mu_\mathrm{eq}}}
\newcommand{\eneq}{{E_\mathrm{eq}}}

\begin{document}

\setcounter{tocdepth}{1}

\title{A variational approach to complex Monge-Amp{\`e}re equations}

\date{\today}

\author{Robert Berman, S{\'e}bastien Boucksom, Vincent Guedj, Ahmed Zeriahi}

\address{Chalmers Techniska H{\"o}gskola \\
 G{\"o}teborg\\
 Sweden}

\email{robertb@math.chalmers.se}

\address{CNRS-Universit{\'e} Paris 7\\
 Institut de Math{\'e}matiques\\
 F-75251 Paris Cedex 05\\
 France}

\email{boucksom@math.jussieu.fr}

\address{LATP, Universit{\'e} Aix-Marseille I\\
13453 Marseille Cedex 13\\
 France}

\email{guedj@cmi.univ-mrs.fr}

\address{I.M.T., Universit{\'e} Paul Sabatier\\
31062 Toulouse cedex 09\\
France}

\email{zeriahi@math.ups-tlse.fr}

\begin{abstract} We show that degenerate complex Monge-Amp\`ere equations in a big cohomology class of a compact K\"ahler manifold can be solved using a variational method independent of Yau's theorem. Our formulation yields in particular a natural pluricomplex analogue of the classical logarithmic energy of a measure. We also investigate K\"ahler-Einstein equations on Fano manifolds. Using continuous geodesics in the closure of the space of K\"ahler metrics and Berndtsson's positivity of direct images we extend Ding-Tian's variational characterization and Bando-Mabuchi's uniqueness result to singular K\"ahler-Einstein metrics. Finally using our variational characterization we prove the existence, uniqueness and convergence as $k\to\infty$ of $k$-balanced metrics in the sense of Donaldson both in the (anti)canonical case and with respect to a measure of finite pluricomplex energy.
\end{abstract} 

\maketitle

\tableofcontents

\newpage

\section*{Introduction}

Solving degenerate complex Monge-Amp{\`e}re equations has been the subject of intensive studies in the past decade, in connection with the search for canonical models and metrics of complex algebraic varieties (see e.g.~\cite{Kol98}, \cite{Tian}, \cite{Che00}, \cite{Don05a}, \cite{Siu06}, \cite{BCHM06}, \cite{EGZ09}, \cite{ST08}).

Many of these results ultimately relied on the seminal work of Yau \cite{Yau78} which involved a continuity method and difficult \emph{a priori} estimates to construct smooth solutions to non-degenerate Monge-Amp\`ere equations. 

But the final goal and outcome of some of these results was to produce singular solutions in degenerate situations, and the main goal of the present paper is to show that one can use the \emph{direct methods} of the calculus of variations to obtain such solutions. Our approach is to some extent a complex analogue of the method used by Aleksandrov to provide weak solutions to the Minkowski problem~\cite{Ale38}, i.e.~the existence of compact convex hypersurfaces of $\R^n$ with prescribed Gaussian curvature. 

Our approach yields more natural proofs of the main results of \cite{GZ07}, \cite{EGZ09}, \cite{BEGZ08}, together with 
several new results to be described below.

\subsection{Weak solutions to Calabi's conjecture and balanced metrics}

\subsubsection{Previous results}
Consider for the moment a compact K{\"a}hler $n$-dimensional manifold $(X,\omega)$ normalized by $\int_X\omega^n=1$. Denote by $\cM_X$ the set of all probability measures on $X$. Given a probability measure $\mu\in\cM_X$ with smooth positive density, it was proved in~\cite{Yau78} that there exists a unique K{\"a}hler form $\eta$ in the cohomology class of $\omega$ such that $\eta^n=\mu$ . More singular measures $\mu\in\cM_X$ were later considered in~\cite{Kol98}. In that case $\eta$ is to be replaced by an element of the set $\cT(X,\omega)$ of all closed positive $(1,1)$-currents $T$ cohomologous to $\omega$, which can thus be written $T=\omega+dd^c\f$ where $\f$ is an $\omega$-psh function, the \emph{potential} of $T$ (defined up to a constant). When $\f$ is bounded the positive measure $T^n$ was defined by Bedford-Taylor, and Ko\l{}odziej showed the existence of a unique $T\in\cT(X,\omega)$ with continuous potential such that $T^n=\mu$ when $\mu$ has $L^{1+\e}$-density. 

In order to consider more singular measures one first needs to extend the Monge-Amp{\`e}re operator $T\mapsto T^n$. Even though this operator cannot be extended in a reasonable way to the whole of $\cT(X,\omega)$, it was shown in~\cite{GZ07,BEGZ08} using a construction of~\cite{BT87} that one can in fact define the \emph{non-pluripolar} product of arbitrary closed positive $(1,1)$-currents $T_1,...,T_p$ on $X$. It yields a closed positive $(p,p)$-current 
$$\langle T_1\wedge...\wedge T_p\rangle$$
putting no mass on pluripolar sets and whose cohomology class is bounded in terms of the cohomology classes of the $T_j$'s only. In particular given $T\in\cT(X,\omega)$ we get a positive measure $\langle T^n\rangle$ putting no mass on pluripolar sets and of total mass 
$$
\int_X\langle T^n\rangle\le\int_X\omega^n=1
$$
Equality holds if $T$ has bounded potential, and more generally currents $T\in\cT(X,\omega)$ for which equality holds are said to have \emph{full Monge-Amp{\`e}re mass}, in which case it is licit to simply write $T^n=\langle T^n\rangle$. Now the main result of~\cite{GZ07} states that every non-pluripolar measure $\mu\in\cM_X$ is of the form $\mu=T^n$ for some $T\in\cT(X,\omega)$ with full Monge-Amp{\`e}re mass, which is furthermore unique as was later shown in this generality in~\cite{Din09}. 

The proofs of the above results from~\cite{Kol98,GZ07} eventually reduce by regularization to the smooth case treated in~\cite{Yau78}. Our first goal in the present article is to show how to solve singular Monge-Amp{\`e}re equations by the direct method of the calculus of variations, independently of~\cite{Yau78}. 

\subsubsection{The variational approach}
Denote by $\cT^1(X,\omega)$ the set of all currents $T\in\cT(X,\omega)$ with full Monge-Amp{\`e}re mass and whose potential is furthermore integrable with respect to $T^n$. According to \cite{GZ07,BEGZ08} currents $T$ in $\cT^1(X,\omega)$ are characterized by the condition $J(T)<+\infty$, where $J$ denotes a natural extension of Aubin's $J$-functional~\cite{Aub84} obtained as follows. One first considers the \emph{Aubin-Mabuchi energy functional} defined on smooth $\omega$-psh functions $\f$ by
$$
E(\f):=\frac{1}{n+1}\sum_{j=0}^n\int_X\f(\omega+dd^c\f)^j\wedge\omega^{n-j}
$$
\cite{Aub84,Mab86}. It is easy
to show using integration by parts that the G{\^a}teaux derivative of $E$ at
$\f$ is given by integration against $(\omega+dd^c\f)^n$. This implies in particular that $E$ is non-decreasing on smooth $\omega$-psh functions and a computation of its second derivative (see
equation~(\ref{equ:eseconde}) below) also shows that $E$ is \emph{concave}. This functional is now extended by monotonicity to arbitrary $\omega$-psh functions by setting
$$E(\f):=\inf\{E(\p)|\,\p\text{ smooth }\omega\text{-psh},\,\p\ge\f\}\in[-\infty,+\infty[,$$
and the $J$-functional is in turn defined by 
$$J(T):=\int_X\f\omega^n-E(\f)$$
for $T=\omega+dd^c\f$. It is well-defined by translation invariance and yields a convex lower semicontinuous function
$$J:\cT(X,\omega)\to[0,+\infty]$$
which induces an exhaustion function on $\cT^1(X,\omega)=\{J<+\infty\}$ in the sense that $\{J\le C\}$ is compact for each $C>0$.  
  
Now observe that the functional $\f\mapsto E(\f)-\int_X\f d\mu$ also descends to a concave functional 
$$F_\mu:\cT^1(X,\omega)\to]-\infty,+\infty]$$
by translation invariance and set
$$E^*(\mu):=\sup_{\cT^1(X,\omega)}F_\mu,$$
This yields a convex lower semicontinuous functional
$$
E^*:\cM_X\to[0,+\infty]
$$
which is essentially the \emph{Legendre transform} of $E$ and will be called the \emph{pluricomplex electrostatic energy}. 

Indeed in case $(X,\omega)$ is $\PP^1$ endowed with its Fubiny-Study metric, $E^*(\mu)$ is equal up to a factor to the logarithmic energy $I(\mu-\omega)$ of the signed measure $\mu-\omega$ with total mass $0$ (cf.~ Section~\ref{sec:pluri}). We shall thus say by analogy that $\mu\in\cM_X$ has \emph{finite energy} iff $E^*(\mu)<+\infty$. 

\smallskip

We can now state our first main result.

\begin{thmA} 
{\it 
A measure $\mu\in\cM_X$ has finite energy iff $\mu=T_\mu^n$ with $T_\mu\in\cT^1(X,\omega)$, which is characterized as the unique maximizer of $F_\mu$ on $\cT^1(X,\omega)$. }
\end{thmA}
We will also show in Corollary \ref{cor:begz} how to recover as a consequence the main result of~\cite{GZ07}. 

\smallskip

The proof of Theorem A splits in two parts. The first one consists in showing that any maximizer $T\in\cT^1(X,\omega) $ of $F_\mu$ has to satisfy $T^n=\mu$, i.e. that a maximizer $\f$ of $E(\f)-\int\f d\mu$ satisfies the Euler-Lagrange equation $(\omega+dd^c\f)^n=\mu$. This is actually non-trivial even when $\f$ is smooth, the difficulty being that the set of $\omega$-psh functions has a boundary, 
so that a maximum is \emph{a priori} not a critical point. This difficulty is overcome by adapting to our case the approach of~\cite{Ale38}. The main technical tool here is the differentiability result of~\cite{BB08}, which is the complex analogue of the key technical result of~\cite{Ale38}. 

The next step in the proof of Theorem A is then to show the \emph{existence} of a maximizer for $F_\mu$ when $\mu$ is assumed to satisfy $E^*(\mu)<+\infty$. Since $J$ is an exhaustion function on $\cT^1(X,\omega)$, a maximizer will be obtained by showing that $F_\mu$ is \emph{proper} with respect to $J$ (i.e. $F_\mu\to-\infty$ as $J\to+\infty$) and that it is upper semi-continuous. The latter property is actually the most delicate part of the proof. 

Conversely it easily follows from the concavity property of $F_{\mu}$ that $\mu$ has finite energy as soon as $\mu=T_{\mu}^n$ with
$T_{\mu} \in \cT^1(X,\omega)$.

\subsubsection{Donaldson's balanced metrics} 
Besides providing a solution by a direct method, the properties of $F_\mu$ also imply that any $F_\mu$-maximizing sequence $T_j\in\cT^1(X,\omega)$ has to converge to $T_\mu$. 

A particularly interesting example of such a maximizing sequence is provided by $\mu$-\emph{balanced metrics} in the sense of~\cite{Don05b}. Here we assume that the cohomology class of $\omega$ is the first Chern class of an ample line bundle $L$, and a metric $e^{-\phi}$ on $L$ is then said to be balanced with respect to $\mu$ if $\phi$ coincides with the Fubiny-Study type metric associated to the $L^2$-scalar product on $H^0(L)$ induced by $\phi$ and $\mu$. We will show:

\begin{thmB}
{\it
 Let $L$ be an ample line bundle and let $\mu$ and $T_\mu\in c_1(L)$ be as in Theorem A. Then there exists a $\mu$-balanced metric $\phi_k$ on $kL$ for each $k$ large enough, and their normalized curvature currents $\frac{1}{k}dd^c\phi_k$ converge towards $T_\mu$ in the weak topology.}
\end{thmB}

The existence of balanced metrics was established in~\cite{Don05b} under a stronger regularity condition for $\mu$. 
The convergence result, suggested in~\cite{Don05b} as an analogue of~\cite{Don01}, was  observed to hold for smooth positive measures 
$\mu$ in~\cite{Kel09} as a direct consequence of the work of Wang \cite{Wan05}.

\subsection{The case of a big class}
Up to now we have assumed that the cohomology class $\{\omega\}\in H^{1,1}(X,\R)$ is K{\"a}hler, but our variational approach works just as well in the more general case of \emph{big} cohomology classes, as considered in~\cite{BEGZ08}. Note that the case of a big class enables in particular to extend our results to the case where $X$ is singular, since the pull-back of a big class to a resolution of singularities remains big. 

The appropriate version of Theorem A will thus be proved in this more general setting, thereby extending~\cite{GZ07} Theorem 4.2 to the case of a big class, and we will show in Corollary \ref{cor:begz} that it implies the main result of~\cite{BEGZ08}. 

\smallskip

The variational approach also applies to K{\"a}hler-Einstein metrics. We will discuss the Fano case separately below, and assume here instead that $X$ is of \emph{general type}, i.e. $K_X$ is a big line bundle. A metric $e^{-\phi}$ on $K_X$ induces a measure $e^{2\phi}$ on $X$, and we can thus consider the functional
$$
\phi\mapsto E(\phi)-\frac{1}{2}\log\int_X e^{2\phi}
$$ 
which descends to a functional
$$
F_+:\cT^1(K_X)\to\R
$$
by translation invariance. We will then show:

\begin{thmC} 
{\it 
Let $X$ be a manifold of general type. Then $F_+$ is upper semicontinuous and $J$-proper. 
It achieves its maximum on $\cT^1(K_X)$ at a unique point $T_{KE}=dd^c\phi_{KE}$ which satisfies 
$$
\langle T_{KE}^n\rangle=e^{2\phi_{KE}+c}
$$
for some $c\in\R$.}
\end{thmC} 

The solution $\phi_{KE}$ therefore coincides with the singular K{\"a}hler-Einstein metric of~\cite{EGZ09, ST08, BEGZ08}, which was proved to have \emph{minimal singularities} in~\cite{BEGZ08}. The ingredients entering the proof of Theorem C are similar to that of Theorem A. The functional $F_+$ is concave by H{\"o}lder's inequality, and we will show that it is upper semicontinuous and $J$-proper. This will show that a maximizer exists, and we then deduce that a maximizer must satisfy the desired equation by the differentiability result of~\cite{BB08}.

\subsection{Singular K{\"a}hler-Einstein metrics on Fano manifolds} 

Assume now that $X$ is a Fano manifold, i.e. $-K_X$ is ample. A psh weight $\phi$ on $-K_X$ with full Monge-Amp{\`e}re mass has zero Lelong numbers, thus $e^{-2\phi}$ can be seen as volume form on $X$ with $L^p$ density for every $p<+\infty$. The functional 
$$\phi\mapsto E(\phi)+\frac{1}{2}\log\int_X e^{-2\phi},$$
descends to 
$$F_-:\cT^1(-K_X)\to\R$$ 
which is Ding-Tian's functional~\cite{Tia97} up to sign. The critical points of $F_-$ in the space of K{\"a}hler forms $\omega\in c_1(X)$ are exactly the 
K{\"a}hler-Einstein metrics. Tian and Ding-Tian obtained the following results~\cite{Tia97,Tian}, assuming that $H^0(T_X)=0$, so that K{\"a}hler-Einstein metrics are unique by~\cite{BM87}: $X$ admits a K{\"a}hler-Einstein $\omega_{KE}$ iff $F_-$ is $J$-proper, and $\omega_{KE}$ is then a maximizer of $F_-$. 

Even though this result is variational in spirit, its actual proof by Ding-Tian relies on the continuity method. 
Using our variational approach we reprove part of this result independently of the continuity method and without any assumption on $H^0(T_X)$.

\begin{thmD} 
{\it 
Let $X$ be a Fano manifold. Then a current $T=dd^c\phi$ in $\cT^1(-K_X)$ is a maximizer of $F_-$ iff it satisfies the K{\"a}hler-Einstein equation
$T^n=e^{-2\phi+c}$
for some $c\in\R$. 

If $F_-$ is $J$-proper the supremum is attained and  there exists $T_{KE}=dd^c \phi_{KE} \in \cT^1(-K_X)$ such that
$T_{KE}^n=e^{-2\phi_{KE}}$. }
\end{thmD}

As we shall see such currents 
automatically have continuous potentials by~\cite{Kol98}. It is an interesting problem to investigate higher regularity of these functions. 

\smallskip

A striking feature of the present situation is that $F_-$ is \emph{not} concave. However  $E$ is \emph{geodesically affine} for the $L^2$-metric on the space of strictly psh weights considered in~\cite{Mab87,Sem92,Don99}, and it follows from Berndtsson's results on psh variation of Bergman kernels~\cite{Bern09a} that $L_-$ is \emph{geodesically convex} with respect to the $L^2$-metric. We thus see that $F_-$ is geodesically concave, which morally explains Ding-Tian's result (compare Donaldson's analogous result for the Mabuchi functional~\cite{Don05a}). 

However a main issue is of course that \emph{smooth} geodesics are not known to exist in general. The proof of Theorem D will instead rely on \emph{continuous} geodesics $\phi_t$, whose existence is easily obtained. 

\smallskip

Using similar ideas we give a new proof of Bando-Mabuchi's uniqueness result \cite{BM87} and extend it to the 
case of singular K{\"a}hler-Einstein currents:

\begin{thmE}
{\it
Let $X$ be a Fano manifold.
Assume that $X$ admits a smooth K{\"a}hler-Einstein metric $\omega_{KE}$ and that $H^0(T_X)=0$. 
Then $\omega_{KE}$ is the unique maximizer of $F_-$ over the whole of $\cT^1(-K_X)$. }
\end{thmE}
An important step in the proof is to show that each $\phi_t$ in the geodesic connecting two K\"ahler-Einstein metrics satisfies the K{\"a}hler-Einstein equation for all $t$ if $\phi_0$ and $\phi_1$ do. Even though the geodesic $\phi_t$ is actually known to be (almost) $\cC^{1,1}$ \cite{Che00,Blo09}, a main technical point is that $\phi_t$ is a priori not \emph{strictly} psh, and one has to resort again to the differentiability result of~\cite{BB08} to infer that $\phi_t$ is K{\"a}hler-Einstein from the fact that it maximizes $F_-$.

Finally we establish in Theorem~\ref{thm:bal} an analogue of Theorem B for K\"ahler-Einstein metrics. More specifically let $X$ be Fano with $H^0(T_X)=0$ and assume that $\omega_{KE}$ is a K\"ahler-Einstein metric. We will show that there exists a unique $k$-anticanonically balanced metric $\omega_k\in c_1(X)$ in the sense of~\cite{Don05b} for each $k\gg 1$ and that $\omega_k\to\omega_{KE}$ weakly. The proof of the existence of such anticanonically balanced metrics relies in a crucial way on the linear growth estimate for $F_-$ established in~\cite{PSSW08}. A proof of these results in the anti-canonically balanced case has been announced in~\cite{Kel09}ÊTheorem 5. The existence and \emph{uniform }Êconvergence of canonically balanced metrics has also been independently been obtain by B.Berndtsson (personal communication). 
 
\newpage

\begin{org} The structure of the paper is as follows.
\begin{itemize}
\item Section 1 is devoted to preliminary results in the big case that are extracted from~\cite{BEGZ08} and~\cite{BD09}. The only new result is the outer regularity of the Monge-Amp\`ere capacity in the big case. 

\item Section 2 is similarly a refresher on energy functionals whose goal is to recall results from~\cite{GZ07,BEGZ08} as well as to extend to the singular case of number of basic properties that are probably well-known in the smooth case. 

\item Section 3 investigates the continuity and growth properties of the functionals defined by integrating quasi-psh functions against a given Borel measure. 

\item Section 4 is devoted to the proof of Theorem A in the general case of big classes. Theorem~\ref{thm:var} and Theorem~\ref{thm:main} are the main statements. 

\item Section 5 connects our pluricomplex energy of measures to more classical notions of capacity and to some of the results of~\cite{BB08}.

\item Section 6 is devoted to singular K\"ahler-Einstein metrics. It contains the proof of Theorems C, D and E. 

\item Finally Section 7 contains our results on balanced metrics. The main result is Theorem~\ref{thm:bal} which treats in parallel the (anti)canonically balanced case and balanced metrics with respect to a singular measure (Theorem B). 
\end{itemize}
\end{org}

\begin{ackn} We would like to thank J.-P.Demailly, P.Eyssidieux, J.Keller and M.Paun for several useful conversations. We are especially grateful to B.Berndtsson for indicating to us that the crucial result of Lemma~\ref{lem:bernd} was a consequence of his positivity results on direct images. 
\end{ackn}

\section{Preliminary results on big cohomology classes}
In this whole section $\theta$ denotes a smooth closed 
$(1,1)$-form on a compact K{\"a}hler manifold $X$. 

\subsection{Quasi-psh functions}
Recall that an upper semi-continuous function $$\f:X\to[-\infty,+\infty[$$
is said to be \emph{$\theta$-psh} iff $\f\in L^1(X)$ and $\theta+dd^c\f\ge 0$ in the sense of currents, where $d^c$ is normalized so that 
$$
dd^c=\frac{i}{\pi}\partial\overline{\partial}.
$$

By the $dd^c$-lemma any closed positive $(1,1)$-current $T$ cohomologous to $\theta$ can conversely be written as $T=\theta+dd^c\f$ for some $\theta$-psh function $\f$ which is furthermore unique up to an additive constant. 

The set of all $\theta$-psh functions $\f$ on $X$ will be denoted by $\cP(X,\theta)$ and endowed with the weak topology, which coincides with the $L^1(X)$-topology. By Hartogs' lemma $\f\mapsto\sup_X\f$ is continuous in the weak topology. Since the set of closed positive currents in a fixed cohomology class is compact (in the weak topology), it follows that the set of $\f\in\cP(X,\theta)$ normalized by $\sup_X\f=0$ is compact. 

We introduce the extremal function $V_\theta$ defined by
\begin{equation}\label{equ:extrem}V_\theta(x):=\sup\{\f(x)|\f\in\cP(X,\theta),\sup_X\f\le 0\}.
\end{equation}
It is a $\theta$-psh function with \emph{minimal singularities} in the sense of Demailly, i.e.~we have 
$\f\le V_\theta+O(1)$ for any $\theta$-psh function $\f$. In fact it is straightforward to see that the following 'tautological maximum principle' holds:
\begin{equation}\label{equ:max}\sup_X\f=\sup_X(\f-V_\theta)
\end{equation}
for any $\f\in\cP(X,\theta)$.

\subsection{Ample locus and non-pluripolar products}
The cohomology class $\{\theta\}\in H^{1,1}(X,\R)$ is said to be \emph{big} iff there exists a closed $(1,1)$-current 
$$T_+=\theta+dd^c\f_+$$ 
cohomologous to $\theta$ such that $T_+$ is \emph{strictly positive} (i.e. $T_+\ge\omega$ for some (small) K{\"a}hler form $\omega$). By Demailly's regularisation theorem~\cite{Dem92}  one can then furthermore assume that $T_+$ has \emph{analytic singularities}, that is there exists $c>0$ such that locally on $X$ we have 
$$\f_+=c\log\sum_{j=1}^N|f_j|^2\text{ mod }C^\infty$$
where $f_1,...,f_N$ are local holomorphic functions. Such a current $T$ is then smooth on a Zariski open subset $\Omega$, and the \emph{ample locus}  $\Amp(\theta)$ of $\theta$ (in fact of its class $\{\theta\}$) is defined as the largest such Zariski open subset (which exists by the Noetherian property of closed analytic subsets). 

Note that \emph{any} $\theta$-psh function $\f$ with minimal singularities is locally bounded on the ample locus $\Amp(\theta)$ since it has to satisfy $\f_+\le\f+O(1)$. 

In~\cite{BEGZ08} the (multilinear) \emph{non-pluripolar product} 
$$(T_1,...,T_p)\mapsto\langle T_1\wedge...\wedge T_p\rangle$$ 
of closed positive $(1,1)$-currents is shown to be well-defined as a closed positive $(p,p)$-current putting no mass on pluripolar sets. 
In particular given $\f_1,...,\f_n\in\cP(X,\theta)$ we define their mixed Monge-Amp{\`e}re measure as 
$$\MA(\f_1,...,\f_n)=\langle(\theta+dd^c\f_1)\wedge...\wedge(\theta+dd^c\f_n)\rangle.$$
It is a non-pluripolar positive measure whose total mass satisfies 
$$\int_X\MA(\f_1,...,\f_n)\le\vol(\theta)$$
where the right-hand side denotes the \emph{volume} of the cohomology class of $\theta$. 
If $\f_1,...,\f_n$ have minimal singularities then they are locally bounded on $\Amp(\theta)$, and the product $$(\theta+dd^c\f_1)\wedge...\wedge(\theta+dd^c\f_n)$$ is thus well-defined by Bedford-Taylor~\cite{BT82}. Its trivial extension to $X$ coincides with $\MA(\f_1,...,\f_n)$, and we have 
$$\int_X\MA(\f_1,...,\f_n)=\vol(\theta).$$
In case $\f_1=...=\f_n=\f$, we simply set
$$\MA(\f)=\MA(\f,...,\f)$$
and we say that $\f$ has \emph{full Monge-Amp{\`e}re mass} iff $\int_X\MA(\f)=\vol(\theta)$. 
We thus see that $\theta$-psh functions with minimal singularities have full Monge-Amp{\`e}re mass, but the converse is not true. 

A crucial point is that the non-pluripolar Monge-Amp{\`e}re operator is continuous along monotonic sequences of functions with full Monge-Amp{\`e}re mass. In fact we have (cf.~\cite{BEGZ08} Theorem 2.17):

\begin{prop}\label{prop:cont} 
The operator
$$
(\f_1,...,\f_n)\mapsto\MA(\f_1,...,\f_n)
$$
is continuous along monotonic sequences of functions with full Monge-Amp{\`e}re mass. 
If $\int_X(\f-V_\theta)\MA(\f)$ is finite, then 
$$
\lim_{j\to\infty}(\f_j-V_\theta)\MA(\f_j)=(\f-V_\theta)\MA(\f)
$$
for any monotonic sequence $\f_j\to\f$.
\end{prop}

\subsection{Regularity of envelopes} 
In case $\{\theta\}\in H^{1,1}(X,\R)$ is a \emph{K{\"a}hler} class, plenty of \emph{smooth} $\theta$-psh functions are available. On the other hand for a general \emph{big} class the existence of even a \emph{single} $\theta$-psh function with minimal singularities that is also $C^\infty$ on the ample locus $\Amp(\theta)$ is unknown. For instance it follows from~\cite{Bou04}  that \emph{no} $\theta$-psh function with minimal singularities will have analytic singularities unless $\{\theta\}$ admits a \emph{Zariski decomposition} (on some birational model of $X$). Examples of big line bundles without  a Zariski decomposition have been constructed by Nakayama (see~\cite{Nak04} Theorem 2.10 P.136). 

On the other hand using Demailly's regularization theorem one can easily show that $V_\theta$ satisfies 
$$V_\theta(x)=\sup\{\f(x)|\f\in\cP(X,\theta)\text{ with analytic singularities},\,\sup_X\f\le 0\}$$
for $x\in\Amp(\theta)$, which implies in particular that $V_\theta$ is in fact \emph{continuous} on $\Amp(\a)$. But we actually have the following much stronger regularity result on the ample locus. It was first obtained by the first named author in~\cite{Ber07} in case $\a=c_1(L)$ for a big line bundle $L$, and the general case is proved in~\cite{BD09}. 

\begin{thm}\label{thm:c11} The function $V_\theta$ has locally bounded Laplacian on $\Amp(\theta)$. 
\end{thm}
Since $V_\theta$ is quasi-psh this result is equivalent to the fact that the curent $\theta+dd^c V_\theta$
has $L^\infty_{loc}$ coefficients on $\Amp(\a)$ and shows in particular by Schauder's elliptic estimates that $V_\theta$ is in fact $C^{2-\e}$ on $\Amp(\a)$ for each $\e>0$. 
 
As was observed in~\cite{Ber07} we also get as a consequence the following nice description of the Monge-Amp{\`e}re measure of $V_\theta$.

\begin{cor}\label{cor:MAreg} The Monge-Amp{\`e}re measure $\MA(V_\theta)$ has $L^\infty$-density with respect to Lebesgue measure. More specifically we have $\theta\ge 0$ pointwise on $\{V_\theta=0\}$ and 
$$\MA(V_\theta)={\bf 1}_{\{V_\theta=0\}}\theta^n.$$
\end{cor}

\subsection{Monge-Amp{\`e}re capacity}
Let $\theta$ be a smooth closed $(1,1)$-form with big cohomology class. As in~\cite{BEGZ08} we define the \emph{Monge-Amp{\`e}re (pre)capacity} in our setting as the upper envelope of all measures $\MA(\f)$ with $\f\in\cP(X,\theta)$, $V_\theta-1\le\f\le V_\theta$, i.e.
\begin{equation}\label{equ:cap}
\ca(B):=\sup \left\{\int_B\MA(\f),\,\f\in\cP(X,\theta),\,V_\theta-1\le\f\le V_\theta\text{ on }X \right\}.
\end{equation}
for every Borel subset $B$ of $X$. In what follows we adapt to our setting some arguments of~\cite{GZ05} Theorem 3.2 (which dealt with the case where $\theta$ is a K{\"a}hler form). 

\begin{lem}\label{lem:achieved} If $K$ is compact the supremum in the definition of $\ca(K)$ is achieved by the usc regularisation of 
$$h_K:=\sup\{\f\in\cP(X,\theta),\,\f\le V_\theta\text{ on }X\text{ and }\f\le V_\theta-1\text{ on }K\}.$$
\end{lem}

\begin{proof} It is clear that $h_K^*$ is a candidate in the supremum defining $\ca(K)$. Conversely pick $\f\in\cP(X,\theta)$ such that $V_\theta-1\le\f\le V_\theta$ on $X$. We have to show that
$$\int_K\MA(\f)\le\int_K\MA(h_K^*).$$
Upon replacing $\f$ by $(1-\e)\f+\e V_\theta$ and then letting $\e>0$ go to $0$ we may assume that $V_\theta-1<\f\le V_\theta$ everywhere on $X$. Noting that $K\subset\{h_K^*<\f\}$ we get
$$\int_K\MA(\f)\le\int_{\{h_K^*<\f+1\}}\MA(\f)$$
$$\le\int_{\{h_K^*<\f+1\}}\MA(h_K^*)$$
by the comparison principle (cf.~\cite{BEGZ08} Corollary 2.3 for a proof in our setting) 
$$\le\int_{\{h_K^*<V_\theta\}}\MA(h_K^*)=\int_K\MA(h_K^*)$$
by Lemma~\ref{lem:supp} below and the result follows.
\end{proof}

\begin{lem}\label{lem:supp} Let $K$ be a compact subset. Then we have $h_K^*=V_\theta-1$ a.e. on $K$ and $h_K^*=V_\theta$ a.e. on $X-K$ with respect to the measure $\MA(h_K^*)$.
\end{lem}
\begin{proof} We have 
$$h_K\le V_\theta-1\le h_K^*\text{ on } K.$$ 
But the set $\{h_K<h_K^*\}$ is pluripolar by Bedford-Taylor's theorem, so it has zero measure with respect to the non-pluripolar measure $\MA(h_K^*)$ and the second point follows. 

On the other hand by Choquet's lemma there exists a sequence of $\theta$-psh functions $\f_j$ increasing a.e.~to $h_K^*$ such that $\f_j\le V_\theta$ on $X$ and $\f_j\le V_\theta-1$ on $K$. If $B$ is a small open ball centered at a point 
$$x_0\in\Amp(\theta)\cap\{h_K^*<V_\theta\}\cap(X-K)$$ 
then we get 
$$h_K\le V_\theta(x_0)-\delta\le V_\theta\text{ on }B$$ 
for some $\de>0$ by continuity of $V_\theta$ on $\Amp(\theta)$ (cf.Theorem~\ref{thm:c11}) and it follows that the function $\widehat{\f_j}$ which coincides with $\f_j$ outside $B$ and satisfies $\MA(\widehat{\f_j})=0$ on $B$ also satisfies 
$$\widehat{\f_j}\le V_\theta(x_0)\le V_\theta\text{ on }B.$$ 
We infer that $\widehat{\f_j}$ increases a.e. to $h_K^*$ and the result follows by Beford-Taylor's continuity theorem for the Monge-Amp{\`e}re along non-decreasing sequences of locally bounded psh functions. 
\end{proof}

By definition, a positive measure $\mu$ is absolutely continuous with respect the capacity $\ca$ iff $\ca(B)=0$ implies $\mu(B)=0$. This means exactly that $\mu$ is non-pluripolar in the sense that $\mu$ puts no mass on pluripolar sets. Since $\mu$ is subadditive, it is in turn equivalent to the existence of a non-decreasing right-continuous function $F:\R_+\to\R_+$ such that 
$$\mu(B)\le F(\ca(B))$$
for all Borel sets $B$. Roughly speaking the speed at which $F(t)\to 0$ as $t\to 0$ measures "how non-pluripolar" $\mu$ is. 

\begin{prop}\label{prop:closed} Let $F:\R_+\to\R_+$ be non-decreasing and right-continuous. Then the convex set of all positive measures $\mu$ on $X$ with $\mu(B)\le F(\ca(B))$ for all Borel subsets $B$ is closed in the weak topology.
\end{prop}
\begin{proof} Since $X$ is compact the positive measure $\mu$ is inner regular, i.e.
$$\mu(B)=\sup_{K\subset B}\mu(K)$$
where $K$ ranges over all compact subsets of $B$. It follows that $\mu(B)\le F(\ca(B))$ holds for every Borel subset $B$ iff $\mu(K)\le F(\ca(K))$ holds for every compact subset $K$. This is however not enough to conclude since $\mu\mapsto\mu(K)$ is \emph{upper} semi-continuous in the weak topology. We are going to show in turn that 
$$\mu(K)\le F(\ca(K))$$ holds for every compact subset $K$ iff 
$$\mu(U)\le F(\ca(U))$$ for every open subset $U$ by showing that
\begin{equation}\label{equ:open}\ca(K)=\inf_{U\supset K}\ca(U)
\end{equation}
where $U$ ranges over all open neighbourhoods of $K$. Indeed since $F$ is right-continuous this yields $F(\ca(K))=\inf_{U\supset K} F(\ca(U))$. But $\mu\mapsto\mu(U)$ is now \emph{lower} semi-continuous in the weak topology so this will conclude the proof of Proposition~\ref{prop:closed}. 

By Lemma~\ref{lem:achieved} and~\ref{lem:supp} 
\begin{equation}\label{equ:comp}
\ca(K)=\int_K\MA(h_K^*)=\int_X(V_X-h_K^*)\MA(h_K^*)
\end{equation}
holds for every compact subset $K$. Now let $K_j$ be a decreasing sequence of compact neighbourhoods of a given compact subset $K$. It is straightforward to check that $h_{K_j}^*$ increases a.e. to $h_K^*$, and Proposition~\ref{prop:cont} thus yields
 $$\inf_{U\supset K}\ca(U)\ge\ca(K)=\lim_{j\to\infty}\ca(K_j)\ge\inf_{U\supset K}\ca(U)$$
as desired. 
\end{proof}

\begin{rem} Since the Monge-Amp{\`e}re precapacity is defined as the upper envelope of a family of Radon measures, it is automatically inner regular, i.e. we have
$$\ca(B)=\sup_{K\subset B}\ca(K)$$
where $K$ ranges over all compact subsets of $B$. On the other hand let $\ca^*$ be the 
outer regularisation of $\ca$, defined on an arbitrary subset $E$ by
$$\ca^*(E):=\inf_{U\supset E}\ca(U).$$
The above argument shows that 
$$\ca^*(K)=\ca(K)$$
holds for every compact subset $K$. Using (\ref{equ:comp}) and following word for word the second half of the proof of Theorem 5.2 in~\cite{GZ05} one can further show that $\ca^*$ is in fact an (outer regular) \emph{Choquet capacity}, and it then follows from Choquet's capacitability theorem that $\ca^*$ is also \emph{inner regular} on Borel sets. We thus get
$$\ca(B)\le\ca^*(B)=\sup_{K\subset B}\ca^*(K)$$
$$=\sup_{K\subset B}\ca(K)\le\ca(B),$$
which means that $\ca$ is also \emph{outer regular} on Borel subsets in the sense that
$$\ca(B)=\inf_{U\supset B}\ca(U).$$
\end{rem}

\section{Finite energy classes}\label{sec:energy}

We let again $\theta$ be a closed smooth $(1,1)$-form with big cohomology class. It will be convenient (and harmless by homogeneity) to assume that the volume is normalised by 
$$
\vol(\theta)=1.
$$
For any $\f_1,...,\f_n\in\cP(X,\theta)$ with full Monge-Amp{\`e}re mass the mixed Monge-Amp{\`e}re measure $\MA(\f_1,...,\f_n)$ is thus a \emph{probability} measure. We will denote $\Omega:=\Amp(\theta)$ the ample locus of $\theta$.

\subsection{Aubin-Mabuchi energy functional}
We define the \emph{Aubin-Mabuchi energy} of $\f\in\cP(X,\theta)$ with minimal singularities by
\begin{equation}\label{equ:energy_var}
E(\f):=\frac{1}{n+1}\sum_{j=0}^n\int_X(\f-V_\theta)\MA\left(\f^{(j)},V_\theta^{(n-j)}\right).
\end{equation}
Note that its restriction $t\mapsto E(t\f+(1-t)\p)$ to line segments is a polynomial map of degree $n+1$. 

Let $\f,\p\in\cP(X,\theta)$ with minimal singularities. It is easy to show by integration by parts (cf.~\cite{BEGZ08,BB08}) that the G{\^a}teaux derivatives are given by
\begin{equation}\label{equ:eprime}E'(\p)\cdot(\f-\p)=\int_X(\f-\p)\MA(\p)\end{equation}
and 
\begin{equation}\label{equ:eseconde}E''(\p)\cdot(\f-\p,\f-\p)=-n\int_\Omega d(\f-\p)\wedge d^c(\f-\p)\wedge(\theta+dd^c\p)^{n-1},\end{equation}
which shows in particular that $E$ is concave. Integration by parts also yields the following properties proved in~\cite{BEGZ08,BB08}. 

\begin{prop}\label{prop:energy_min} $E$ is concave and non-decreasing. For any $\f,\p\in\cP(X,\theta)$ with minimal singularities we have

\begin{equation}\label{equ:diff}E(\f)-E(\p)=\frac{1}{n+1}\sum_{j=0}^n\int_X(\f-\p)\MA\left(\f^{(j)},\p^{(n-j)}\right)\end{equation}

and

\begin{equation}\label{equ:monotone}\int_X(\f-\p)\MA(\f)\le...\le\int_X(\f-\p)\MA\left(\f^{(j)},\p^{(n-j)}\right)\le...\le\int_X(\f-\p)\MA(\p).\end{equation}
 for $j=0,...,n$.

\end{prop}
We also remark that $E(V_\theta)=0$ and $E$ satisfies the scaling property
\begin{equation}\label{equ:scaling}
E(\f+c)=E(\f)+c
\end{equation}
for any constant $c\in\R$. 

We now introduce the analogue of Aubin's $I$ and $J$-functionals (cf.~\cite{Aub84} P.145,\cite{Tian} P.67). We introduce the \emph{symmetric} expression
$$I(\f,\p):=\int_X(\f-\p)(\MA(\p)-\MA(\f))=-(E'(\f)-E'(\p))\cdot(\f-\p),$$
and we set
$$J_\p(\f):=E(\p)-E(\f)+\int_X(\f-\p)\MA(\p)$$
$$=(E(\p)+E'(\p)\cdot(\f-\p))-E(\f)\ge 0$$
which controls the second order behaviour of $E$ at $\p$ and is non-negative by concavity of $E$. Note that $J_\p$ is convex and non-negative by concavity of $E$. For $\p=V_\theta$ we simply write $J:=J_{V_\theta}$. By concavity of $E$ we have $0\le J_\p(\f)\le I(\f,\p)$. On the other hand Proposition~\ref{prop:energy_min} shows that $E(\f)-E(\p)$ is the mean value of a non-decreasing sequence whose extreme values are $\int_X(\f-\p)\MA(\f)$ and $\int_X(\f-\p)\MA(\p)$, and it follows for elementary reasons that
\begin{equation}\label{equ:compare}\frac{1}{n+1}I(\f,\p)\le J_\p(\f)\le I(\f,\p).
\end{equation}

Elementary algebraic identities involving integration by parts actually show as in~\cite{Tian} P.58 that
\begin{equation}\label{equ:delta}J_\p(\f)=\sum_{j=0}^{n-1}\frac{j+1}{n+1}\int_\Omega d(\f-\p)\wedge d^c(\f-\p)\wedge(\theta+dd^c\p)^j\wedge(\theta+dd^c\f)^{n-1-j}.
\end{equation}
and 
\begin{equation}\label{equ:I}I(\f,\p)=\sum_{j=0}^{n-1}\int_\Omega d(\f-\p)\wedge d^c(\f-\p)\wedge(\theta+dd^c\f)^j\wedge(\theta+dd^c\psi)^{n-1-j}.
\end{equation}

As opposed to $I(\f,\p)$ the expression $J_\p(\f)$ is not symmetric in $(\f,\p)$. However we have
\begin{lem}\label{lem:quasisym} For any two $\f,\p\in\cP(X,\theta)$ with minimal singularities we have
$$n^{-1}J_\p(\f)\le J_\f(\p)\le n J_\p(\f).$$
\end{lem}
\begin{proof} By Proposition~\ref{prop:energy_min} we have
$$n\int_X(\f-\p)\MA(\f)+\int_X(\f-\p)\MA(\p)\le(n+1)\left(E(\f)-E(\p)\right)$$
$$\le \int_X(\f-\p)\MA(\f)+n\int_X(\f-\p)\MA(\p)$$
and the result follows immediately.
\end{proof}

\begin{prop}\label{prop:enc11} For any $\f,\p\in\cP(X,\theta)$ with minimal singularities and any $0\le t\le 1$ we have
$$I(t\f+(1-t)\p,\p)\le nt^2 I(\f,\p).$$
\end{prop}
\begin{proof} We expand out
$$E'(t\f+(1-t)\p)\cdot(\f-\p)=\int_X(\f-\p)\MA(t\f+(1-t)\p)$$
$$=(1-t)^n\int_X(\f-\p)\MA(\p)+\sum_{j=1}^n{n\choose
  j}t^j(1-t)^{n-j}\int_X(\f-\p)\MA(\f^{(j)},\p^{(n-j)})$$
$$\ge(1-t)^n\int_X(\f-\p)\MA(\p)+(1-(1-t)^n)\int_X(\f-\p)\MA(\f)$$
by (\ref{equ:monotone})
$$=(1-t)^nE'(\p)\cdot(\f-\p)+(1-(1-t)^n)E'(\f)\cdot(\f-\p).$$
This yields
$$I(t\f+(1-t)\p,\p)\le t(1-(1-t)^n) I(\f,\p)$$
and the result follows by convexity of $(1-t)^n$.
\end{proof}
Note that by definition of $I$ and $J$ we have
$$\lim_{t\to 0_+}\frac{2}{t^2}J_\p(t\f+(1-t)\p)=\lim_{t\to_0+}\frac{1}{t^2}I(t\f+(1-t)\p),\p)$$
$$=-E''(\p)\cdot(\f-\p,\f-\p).$$

\subsection{Finite energy classes}

As in~\cite{BEGZ08} Definition 2.9 it is natural to extend $E(\f)$ by monotonicity to an \emph{arbitrary} $\f\in\cP(X,\theta)$ by setting
\begin{equation}\label{equ:extend} 
E(\f):=\inf\{E(\p)|\p\in\cP(X,\theta)\text{ with minimal singularities},\,\p\ge\f\}.
\end{equation}

By~\cite{BEGZ08} Proposition 2.10 we have
\begin{prop}\label{prop:energy_gen} The extension
 $$E:\cP(X,\theta)\to[-\infty,+\infty[$$ 
 so defined is concave, non-decreasing and usc.
\end{prop}
As a consequence $E$ is continuous along decreasing sequences, and $E(\f)$ can thus be more concretely obtained as the limit of $E(\f_j)$ for any sequence of $\f_j\in\cP(X,\theta)$ with minimal singularities such that $\f_j$ decreases to $\f$ pointwise. One can for instance take $\f_j=\max(\f,V_\theta-j)$. 

Following~\cite{Ceg98} and~\cite{GZ07} we introduce
\begin{defi}\label{defi:finite} 
The domain of $E$ is denoted by
$$
\cE^1(X,\theta):=\{\f\in\cP(X,\theta),E(\f)>-\infty\}
$$
and its image in the set $\cT(X,\theta)$ of all positive currents cohomologous to $\theta$ will be denoted by $\cT^1(X,\theta)$. For each $C>0$ we also introduce
$$
\cE_C:=\{\f\in\cE^1(X,\theta),\sup_X\f\le 0, E(\f)\ge-C\}.
$$
\end{defi}
Note that $\cE^1(X,\theta)$ and each $\cE_C$ are convex subsets of $\cP(X,\theta)$. 

\begin{lem}\label{lem:compact} For each $C>0$ $\cE_C$ is compact and convex. 
\end{lem}
\begin{proof} Convexity follows from concavity of $E$. Pick $\f\in\cP(X,\theta)$ with $\sup_X\f\le 0$. We then have $\f\le V_\theta$ by (\ref{equ:max}) and it follows from the definition (\ref{equ:energy_var}) of $E$ that 
$$E(\f)\le\int_X(\f-V_\theta)\MA(V_\theta)\le\sup_X\f$$
by (\ref{equ:max}) again. Since $E$ is usc we thus see that $\cE_C$ is a closed subset of the compact set
$$\{\f\in\cP(X,\theta), -C\le\sup_X\f\le 0\}$$
and the result follows.
\end{proof}

\begin{lem}\label{lem:mixte} The integral 
$$\int_X(\f_0-V_\theta)\MA(\f_1,...,\f_n)$$
is finite for every $\f_0,...,\f_n\in\cE^1(X,\theta)$ and is furthermore uniformly bounded for $\f_0,...,\f_n\in\cE_C$. 
\end{lem}
\begin{proof} Upon passing to the canonical approximants, we may assume that $\f_0,...,\f_n$ have minimal singularities. Set $\psi:=\frac{1}{n+1}(\f_0+...+\f_n)$. Observe that $V_\theta-\f_0\le(n+1)(V_\theta-\psi)$. Using the convexity of $-E$ it follows that
$$\int_X(V_\theta-\f_0)\MA(\psi)\le(n+1)\int_X(V_\theta-\psi)\MA(\p)$$
$$\le (n+1)^2|E(\p)|\le(n+1)(|E(\f_0)|+...+|E(\f_n)|).$$
On the other hand expanding out we easily get
$$\MA(\psi)\ge C_n\MA(\f_1,...,\f_n)$$
for some  coefficient $C_n$ only depending on $n$ and the result follows. 
\end{proof}

The following characterization of functions in $\cE^1(X,\theta)$ follows from~\cite{BEGZ08} Proposition 2.11. 

\begin{prop}\label{prop:energy} 
Let $\f\in\cP(X,\theta)$. The following properties are equivalent:
\begin{itemize}
\item $\f\in\cE^1(X,\theta)$.
\item $\f$ has full Monge-Amp{\`e}re mass and $\int_X(\f-V_\theta)\MA(\f)$ is finite.
\item We have
$$\int^{+\infty}dt\int_{\{\f=V_\theta-t\}}\MA(\max(\f,V_\theta-t))<+\infty.$$
\end{itemize}
\end{prop}

Functions in $\cE^1(X,\theta)$ can almost be characterised in terms of the capacity decay of sublevel sets:

\begin{lem}\label{lem:criterion} 
Let $\f\in\cP(X,\theta)$. If 
$$\int_{t=0}^{+\infty}t^n\ca\{\f<V_\theta-t\}dt<+\infty$$
then $\f\in\cE^1(X,\theta)$. Conversely for each $C>0$ 
$$
\int_{t=0}^{+\infty}t\ca\{\f<V_\theta-t\}dt
$$
is bounded uniformly for $\f\in\cE_C$. 
\end{lem}
Note that if $\f$ is an arbitrary $\theta$-psh function then $\ca\{\f<V_\theta-t\}$ usually decreases no faster that $1/t$ as $t\to+\infty$. 

\begin{proof} 
The proof is adapted from Lemma 5.1 in~\cite{GZ07}. 
Observe that for each $t\ge 1$ the function $\f_t:=\max(\f,V_\theta-t)$ satisfies 
$V_\theta-t\le\f_t\le V_\theta$
thus 
$$
t^{-1}\f_t+(1-t^{-1})V_\theta
$$ 
is a candidate in the supremum defining $\ca$, so that
$$
\MA(\f_t)\le t^n\ca.
$$
Now the first assertion follows from Proposition~\ref{prop:energy}.

In order to prove the converse we apply the comparison principle. Pick a candidate 
$$
\p\in\cP(X,\theta),\,V_\theta-1\le\p\le V_\theta
$$
in the supremum defining $\ca$. For $t\ge 1$ we have
$$
\{\f<V_\theta-2t\}\subset\{t^{-1}\f+(1-t^{-1})V_\theta<\p-1\}\subset\{\f<V_\theta-t\}
$$
thus the comparison principle (cf.~\cite{BEGZ08} Corollary 2.3) implies 
$$
\int_{\{\f<V_\theta-2t\}}\MA(\p)\le\int_{\{\f<V_\theta-t\}}\MA(t^{-1}\f+(1-t^{-1})V_\theta)
$$
$$\le\int_{\{\f<V_\theta-t\}}\MA(V_\theta)+\sum_{j=1}^n{n\choose
  j}t^{-j}\int_{\{\f<V_\theta-t\}}\MA\left(\f^{(j)},V_\theta^{(n-j)}\right)$$
$$
\le\int_{\{\f<V_\theta-t\}}\MA(V_\theta)+C_1t^{-1}\sum_{j=1}^n\int_{\{\f<V_\theta-t\}}\MA\left(\f^{(j)},V_\theta^{(n-j)}\right)
$$
since $t\ge 1$ and it follows that
$$
\int_{t=0}^{+\infty}t\ca\{\f<V_\theta-t\}\le C_2+C_3\int_X(V_\theta-\f)^2\MA(V_\theta)
$$
since $E(\f)\ge -C$ and $\ca\le 1$. But $\MA(V_\theta)$ has $L^\infty$-density with respect to Lebesgue measure by Corollary~\ref{cor:MAreg} and it follows from the uniform version of Skoda's theorem~\cite{Zeh01} that there exists $\e>0$ and $C_1>0$ such that 
$$
\int_Xe^{-\e\f}\MA(V_\theta)\le C_1
$$
for all $\f$ in the compact subset $\cE_C$ of $\cP(X,\theta)$. This implies in turn that $\int_X(V_\theta-\f)^2\MA(V_\theta)$ is uniformly bounded for $\f\in\cE_C$ and the result follows. 
\end{proof}

\begin{rem}Proposition B of~\cite{BGZ08b} says that the exponent $n$
  is optimal for the similar statement in the setting of psh functions
  on hyperconvex domains.
\end{rem}

\begin{cor}\label{cor:pluripol} If $A\subset X$ is a (locally) pluripolar subset, then there exists $\f\in\cE^1(X,\theta)$ such that $A\subset\{\f=-\infty\}$. 
\end{cor}

\begin{proof} Since $\{\theta\}$ is big there exists a proper modification $\mu:X'\to X$ and an effective $\R$-divisor $E$ on $X'$ such that $\mu^*\theta-E$ is cohomologous to a K{\"a}hler form $\omega$ on $X'$. By the K{\"a}hler version of Josefson's theorem (\cite{GZ05} Theorem 6.2) we may thus find a positive current $T$ in the class of $\omega$ whose polar set contains $A$. The push-forward $\mu_*(T+E)$ is then a positive current in the class of $\theta$, and we have thus found $\f\in\cP(X,\theta)$ such that $A\subset\{\f=-\infty\}$. Now let $\chi:\R\to\R$ be a smooth convex non-decreasing function such that $\chi(-\infty)=-\infty$ and $\chi(s)=s$ for all $s\ge 0$. If $\f$ is $\theta$-psh, then so is 
$$\f_\chi:=\chi\circ(\f-V_\theta)+V_\theta,$$ 
and $A$ is contained in the poles of $\f_\chi$. On the other hand we can clearly make $\ca\{\f_\chi<V_\theta-t\}$ tend to $0$ as fast as we like when $t\to\infty$ by choosing $\chi$ with a sufficiently slow decay at $-\infty$. It thus follows from Lemma~\ref{lem:criterion} that $\f_\chi\in\cE^1(X,\theta)$ for an appropriate choice of $\chi$, and the result follows. Actually $\chi(t)=-\log(1-t)$ is enough (compare~\cite{GZ07} Example 5.2). 
\end{proof}

\section{Action of a measure on psh functions}

\subsection{Finiteness}

Given a probability measure $\mu$ on $X$ and $\f\in\cP(X,\theta)$ we set
\begin{equation}\label{equ:lmu}
L_\mu(\f):=\int_\Omega(\f-V_\theta)d\mu
\end{equation}
where $\Omega:=\Amp(\theta)$ denotes the ample locus. Since $\Omega$ is Zariski open, we have
$$L_\mu(\f)=\int_X(\f-V_\theta)d\mu$$
if $\mu$ is non-pluripolar. 

This defines a functional $L_\mu:\cP(X,\theta)\to[-\infty,+\infty[$ which is obviously affine and satisfies the scaling property
$$L_\mu(\f+c)=L_\mu(\f)+c$$
for any $c\in\R$.  

In the special case where $\mu=\MA(V_\theta)$ we will write as a short-hand
\begin{equation}\label{equ:izero}
L_0(\f):=L_{\MA(V_\theta)}(\f)=\int_\Omega(\f-V_\theta)\MA(V_\theta)
\end{equation}
so that
$$
J=L_0-E
$$
holds by definition.

\begin{lem}\label{lem:usc} $L_\mu$ is usc on $\cP(X,\theta)$. On the other hand given $\f\in\cP(X,\theta)$ the map $\mu\mapsto L_\mu(\f)$ is also usc. 
\end{lem}
\begin{proof} Let $\f_j\to\f$ be a convergent sequence of functions in
  $\cP(X,\theta)$. Hartogs' lemma implies that $\f_j$ is uniformly bounded from
  above, hence so is $\f_j-V_\theta$. Since we have
$$\f=(\limsup_{j\to\infty}\f_j)^*\ge\limsup_{j\to\infty}\f_j$$
everywhere on $X$ we infer 
$$L_\mu(\f)\ge\limsup_{j\to\infty} L_\mu(\f_j)$$
as desired by Fatou's lemma. The second assertion follows directly from the fact that $\f-V_\theta$ is usc on $\Omega$ since $V_\theta$ is continuous on $\Omega$. 
\end{proof} 

\begin{lem}  \label{lem:2eq4}
Let $\f\in\cP(X,\theta)$ and set $\mu:=\MA(\f)$. 
\begin{enumerate}
\item[(i)] If $\f$ has minimal singularities then $L_\mu$ is finite on $\cP(X,\theta)$. 
\item[(ii)] If $\f\in\cE^1(X,\theta)$ then $L_\mu$ is finite on $\cE^1(X,\theta)$. 
\end{enumerate}
\end{lem}
\begin{proof} (ii) follows directly from Lemma~\ref{lem:mixte}. We prove (i). Let $\p\in\cP(X,\theta)$. We can assume that $\p\le 0$, or equivalently $\p\le V_\theta$. Assume first that $\p$ also has minimal singularities. If we set $\Omega:=\Amp(\theta)$, then we can integrate by parts using Theorem 1.14 of~\cite{BEGZ08} to get
$$\int_\Omega(V_\theta-\p)(\theta+dd^c\f)^n=\int_\Omega(V_\theta-\p)(\theta+dd^cV_\theta)\wedge(\theta+dd^c\f)^{n-1}$$
$$+\int_\Omega(\f-V_\theta)dd^c(V_\theta-\p)\wedge(\theta+dd^c\f)^{n-1}.$$
The second term is equal to 
$$\int_\Omega(\f-V_\theta)(\theta+dd^cV_\theta)\wedge(\theta+dd^c\f)^{n-1}-\int_\Omega(\f-V_\theta)(\theta+dd^c\p)\wedge(\theta+dd^c\f)^{n-1}$$
and each of these terms is controled by 
$$\sup_X|\f-V_\theta|.$$
By iterating integration by parts as above we thus get
$$\int_X(V_\theta-\p)\MA(\f)\le 2n\sup_X|\f-V_\theta)+\int_X(V_\theta-\p)\MA(V_\theta).$$
The result follows by replacing $\p$ by $\max(\p,V_\theta-k)$ and letting $k\to\infty$, since $\MA(V_\theta)$ has $L^\infty$ density with respect to Lebesgue measure. 
\end{proof}

\subsection{Properness and coercivity} 
The $J$-functional is translation invariant thus it descends to a
non-negative, convex and lower semicontinuous function
$J:\cT(X,\theta)\to[0,+\infty]$ which is finite precisely on
$\cT^1(X,\theta)$. It actually defines an \emph{exhaustion} function of
$\cT^1(X,\theta)$:

\begin{lem}\label{lem:Jproper} The function $J:\cT^1(X,\theta)\to[0,+\infty[$ is an exhaustion of $\cT^1(X,\theta)$ in the sense that each sublevel set $\{J\le C\}\subset\cT^1(X,\theta)$ is compact.
\end{lem} 
\begin{proof} By Lemma~\ref{lem:2eq4} there exists $A>0$ such that 
$$\sup_X\f-A\le\int_X\f\MA(V_\theta)\le\sup_X\f.$$
Now pick $T\in\{J\le C\}$ and write it as $T=\theta+dd^c\f$ with $\sup_X\f=0$. We then have 
$$J(T)=\int_X\f\MA(V_\theta)-E(\f)\le C$$
thus $E(\f)\ge -C-A$. This means that the closed set $\{J\le C\}$ is contained in the image of $\cE_{C+A}$ by the quotient map 
$$\cP(X,\theta)\to\cT(X,\theta).$$
The result now follows since $\cE_{C+A}$ is compact by Lemma~\ref{lem:compact}. \end{proof}

The following statement extends part of~\cite{GZ07} Lemma 2.11.

\begin{prop}\label{prop:square} Let $L:\cP(X,\theta)\to[-\infty,+\infty[$ be a convex and non-decreasing function satisfying the scaling property $L(\f+c)=L(\f)+c$.
\begin{enumerate} 
\item[(i)] If $L$ is finite on a given compact convex subset $\cK$ of $\cP(X,\theta)$ then $L$ is bounded on $\cK$. 
\item[(ii)]  If $L$ is finite on $\cE^1(X,\theta)$ then 
\begin{equation}\label{equ:root}\sup_{\cE_C}|L|=O(C^{1/2})
\end{equation}
as $C\to+\infty$.
\end{enumerate}
\end{prop}
\begin{proof} (i) There exists $C>0$ such that 
$$\sup_X(\f-V_\theta)=\sup_X\f\le C$$
for all $\f\in\cK$, thus $L$ is uniformly bounded above by $L(V_\theta)+C$. Assume by contradiction that $L(\f_j)\le -2^j$ for some sequence $\f_j\in\cK$. We then consider $\f:=\sum_{j\ge 1}2^{-j}\f_j$, which belongs to $\cK$ by Lemma~\ref{lem:basic1} below. By (\ref{equ:max}) we have
$$\f\le\sum_{j=1}^N 2^{-j}\f_j+2^{-N} V_\theta$$
 for each $N$, and the right-hand side is a (finite) convex combination of elements in $\cP(X,\theta)$. The properties of $L$ thus imply
$$-\infty<L(\f)\le\sum_{j=1}^N 2^{-j}L(\f_j)+2^{-N}L(V_\theta)=-N+2^{-N}L(V_\theta)$$
and we reach a contradiction by letting $N\to+\infty$.

(ii) By (i) we have $\sup_{\cE_C}|L|<+\infty$ for all $C>0$. Note also that $L(\f)\le L(V_\theta)$ for $\f\in\cE_C$. If $\sup_{\cE_C}|L|=O(C^{1/2})$ fails as $C\to+\infty$, then there exists a sequence $\f_j\in\cE^1(X,\theta)$ with $\sup_X\f_j=0$ such that 
$$t_j:=|E(\f_j)|^{-1/2}\to 0$$ 
and 
\begin{equation}\label{equ:infty}t_jL(\f_j)\to-\infty.
\end{equation}

We claim that there exists $C>0$ such that for any $\f\in\cP(X,\theta)$ with $\sup_X\f=0$ and $t:=|E(\f)|^{-1/2}\le 1$ we have 
$$E(t\f+(1-t)V_\theta)\ge -C.$$ 
Indeed $\int_X(\f-V_\theta)\MA(V_\theta)$ is uniformly bounded when $\sup_X\f=0$ (for instance by (i)) and the claim follows from Proposition~\ref{prop:enc11} applied to $\p=V_\theta$. 

As a consequence we get $t_j\f_j+(1-t_j)V_\theta\in\cE_C$ for all $j\gg 1$, and the convexity property of $L$ thus yield 
$$t_jL(\f_j)+(1-t_j)L(V_\theta)\ge L(t_j\f_j+(1-t_j)V_\theta)\ge\inf_{\cE_C}L>-\infty$$ 
which contradicts (\ref{equ:infty}). 
\end{proof}

\begin{lem}\label{lem:basic1} Let $\f_j\in\cK$ be a sequence in a compact convex subset of $\cP(X,\theta)$. Then $\f:=\sum_{j\ge 1}2^{-j}\f_j$ belongs to $\cK$. 
\end{lem}
\begin{proof} By Hartogs' lemma $\sup_X\f$ is uniformly bounded for $\f\in\cK$, thus we may assume upon translating by a constant that $\sup_X\f\le 0$ for each $\f\in\cK$. Let $\mu$ be a smooth volume form on $X$. Then $\int_X\f_j d\mu$ is uniformly bounded since $\cK$ is a compact subset of $L^1(X)$. It thus follows that $\int_X\f d\mu$ is finite by Fatou's lemma. But since $\f$ is a decreasing limit of functions in $\cP(X,\theta)$ we either have $\f\in\cP(X,\theta)$ or $\f\equiv-\infty$ and the latter case is excluded by $\int_X\f d\mu>-\infty$, qed.
\end{proof}
We will now interpret Proposition~\ref{prop:square} as a \emph{coercivity} condition.
Since our convention is to \emph{maximize} certain functionals in our
variational approach, we shall use the following terminology.

\begin{defi}\label{defi:proper} A function $F:\cT^1(X,\theta)\to\R$ will be said to be
\begin{enumerate}
\item[(i)] $J$-\emph{proper} if $F\to-\infty$ as $J\to+\infty$. 
\item[(ii)] $J$-\emph{coercive} if there exists $\e>0$ and $A>0$ such that
$$
F\le-\e J+A
$$
on $\cT^1(X,\theta)$.
\end{enumerate}
\end{defi}
The function $F$ on $\cT^1(X,\theta)$ is induced
by a function on $\cE^1(X,\theta)$ of the form $E-L$ where $L$
satisfies as above the scaling property. The $J$-coercivity of
$F$ reads
$$
E-L\le-\e(L_0-E)+A
$$
where $\e>0$ can of course be assumed to satisfy $\e<1$ since $J\ge
0$. Since we have
$$
L_0(\f)=\sup_X\f+O(1)
$$
uniformly for $\f\in\cP(X,\theta)$ the $J$-coercivity of $F$ is then easily seen to be equivalent to the growth condition
\begin{equation}\label{equ:growth}
\sup_{\cE_C}|L|\le(1-\e)C+O(1)  
\end{equation}
as $C\to+\infty$.

As a consequence of Proposition~\ref{prop:square} we get
\begin{cor}\label{cor:proper} Let $L:\cE^1(X,\theta)\to\R$ be a convex
  non-decreasing function satisfying the scaling property. Then the
  function $F$ on $\cT^1(X,\theta)$ induced by $E-L$ is $J$-coercive.
\end{cor}

When $X$ is a Fano manifold with $H^0(T_X)=0$ it was shown
in~\cite{Tia97} that $X$ admits a K\"ahler-Einstein metric iff the
function $F_-$ induced on $\cT^1(X,-K_X)$ by the translation invariant function $$
\phi\mapsto E(\phi)+\frac{1}{2}\log\int_X e^{-2\phi}
$$
is $J$-proper (cf.~Section~\ref{sec:KE}). This result was later refined in~\cite{PSSW08} who
showed that $F_-$ is actually $J$-coercive in the above sense if $X$
is K\"ahler-Einstein. The latter result will play a crucial role in
our proof of the existence of anticanonically balanced metrics in
Section~\ref{sec:balanced}. 

Let us finally record the following useful elementary fact.
\begin{prop}\label{prop:proper} Let $F$ be a $J$-proper and usc
  function on $\cT^1(X,\theta)$. Then $F$ achieves its supremum on
  $\cT^1(X,\theta)$. Moreover any asymptotically maximizing sequence
  $T_j\in\cT^1(X,\theta)$ (i.e. such that 
$\lim_{j\to\infty}F(T_j)=\sup F$) stays in a compact subset of
$\cT^1(X,\theta)$ and any accumulation point $T$ of the $T_j$'s is a $F$-maximizer.  
\end{prop}
\begin{proof} Let us recall the standard argument. It is clearly
  enough to settle the second part. Let thus $T_j$ be a maximizing
  sequence. It follows in particular that $F(T_j)$ is bounded from
  below, and the $J$-properness of $F$ thus yields $C>0$ such that
$T_j\in\{J\le C\}$ for all $j$. Since $\{J\le C\}$ is compact there
exists an accumulation point $T$ of the $T_j$'s, and $F(T_j)\to\sup F$
implies $F(T)\ge\sup F$ since $F$ is usc. 
\end{proof}
\subsection{Continuity}
In order to investigate the upper semi-continuity of $F_\mu=E-L_\mu$ on $\cE^1(X,\theta)$ we will use the following general criterion. 

\begin{thm} Let $\mu$ be a non-pluripolar measure and let $\cK\subset\cP(X,\theta)$ be a compact convex subset such that $L_\mu$ is finite on $\cK$. The following properties are equivalent.
\begin{enumerate}
\item[(i)] $L_\mu$ is continuous on $\cK$. 
\item[(ii)] The map $T:\cK\to L^1(\mu)$ defined by $T(\f):=\f-V_\theta$ is continuous. 
\item[(iii)] The set $T(\cK)\subset L^1(\mu)$ is uniformly integrable, i.e. 
$$\int_{t=k}^{+\infty}\mu\{\f\le V_\theta-t\}dt\to 0$$ 
as $k\to+\infty$, uniformly for $\f\in\cK$. 
\end{enumerate}
\end{thm}
\begin{proof} By the Dunford-Pettis theorem, asumption (iii) means that $T(\cK)$ is relatively compact in the weak topology (induced by $L^\infty(\mu)=L^1(\mu)^*$). 

As a first general remark, we claim that graph of $T$ is closed. Indeed let $\f_j\to\f$ be a convergent sequence in $\cK$ and assume that $T(\f_j)\to f$ in $L^1(\mu)$. We have to show that $f=T(\f)$. But $\f_j\to\f$ implies that 
$$\f=(\limsup_{j\to\infty}\f_j)^*$$
everywhere on $X$ by general properties of psh functions. On the other hand the set of points where $(\limsup_{j\to\infty}\f_j)^*>\limsup_{j\to\infty}\f_j$ is negligible hence pluripolar by a theorem of Bedford-Taylor, thus has $\mu$-measure $0$ by assumption on $\mu$. We thus see that $\f=\limsup_j\f_j$ $\mu$-a.e, hence $T(\f)=\limsup_j T(\f_j)$ $\mu$-a.e. Since $T(\f_j)\to f$ in $L^1(\mu)$ there exists a subsequence such that $T(\f_j)\to f$ $\mu$-a.e., and it follows that $f=T(\f)$ $\mu$-a.e. as desired. 

This closed graph property implies that the convex set $T(\cK)$ is closed in the norm topology (hence also in the weak topology by the Hahn-Banach theorem). Indeed if $T(\f_j)\to f$ holds in $L^1(\mu)$, then we may assume that $\f_j\to\f$ in $\cK$ by compactness of the latter space, hence $f=T(\f)$ belongs to $T(\cK)$ by the closed graph property. 

We now prove the equivalence between (i) and (ii). Observe that there exists $C>0$ such that $T(\f)=\f-V_\theta\le C$ for all $\f\in\cK$ since $\sup_X\f=\sup_X(\f-V_\theta)$ is bounded on the compact set $\cK$ by Hartogs' lemma. Given a convergent sequence $\f_j\to\f$ in $\cK$ we have $T(\f)\ge\limsup_{j\to\infty}T(\f_j)$ $\mu$-a.e. as was explained above, thus Fatou's lemma (applied to $C-T(\f_j)\ge 0$) yields the asymptotic lower bound
$$\int T(\f)d\mu\ge\limsup_{j\to\infty}\int T(\f_j)d\mu,$$
and the asymptotic equality case
$$\int T(\f)d\mu=\lim_{j\to\infty}\int T(\f_j)d\mu$$
holds iff $T(\f_j)\to T(\f)$ in $L^1(\mu)$. This follows from a basic lemma in integration theory, 
which proves the desired equivalence. 

If (ii) holds, then the closed convex subset $T(\cK)$ is compact in the norm topology, hence also weakly compact, and (iii) holds by the Dunford-Pettis theorem recalled above. 

Conversely assume that (iii) holds. We will prove (i). Let $\f_j\to\f$ be a convergent sequence in $\cK$. We are to prove that $\int T(\f_j)d\mu\to\int T(\f)d\mu$ in $L^1(\mu)$. We may assume that $\int T(\f_j)d\mu\to L$ for some $L\in\R$ since $T(\cK)$ is bounded, and we have to show that $L=\int T(\f)d\mu$. For each $k$ consider the closed convex envelope
$$\cC_k:=\overline{\conv\{T(\f_j),j\ge k\}}.$$
Each $\cC_k$ is also weakly closed by the Hahn-Banach theorem, hence weakly compact since it is contained in $T(\cK)$. Since $(\cC_k)_k$ is a decreasing sequence of compact subsets there exists $f\in\cap_k\cC_k$. For each $k$ we may thus find a finite convex combination $\psi_k\in\conv\{\f_j,j\ge k\}$ such that $T(\psi_k)\to f$ in the norm topology. Since $\f_j\to\f$ in $\cK$ we also have $\psi_k\to\f$ in $\cK$, hence $f=T(\f)$ by the closed graph property.  On the other hand $\int T(\psi_k)d\mu$ is a convex combination of elements of the form $\int T(\f_j)d\mu$, $j\ge k$, thus $\int T(\psi_k)d\mu\to L$, and we finally get $\int T(\f)d\mu=\int f d\mu=L$ as desired. 
\end{proof}

By H{\"o}lder's inequality a bounded subset of $L^2(\mu)$ is uniformly integrable, hence the previous result applies to yield:

\begin{cor}\label{cor:domcap} Let $\mu$ be a probability measure such that 
$$\mu\le A\ca$$
for some $A>0$. Then $T(\cE_C)$ is bounded in $L^2(\mu)$, and $L_\mu$
is thus continuous on $\cE_C$ for each $C>0$. 
\end{cor}
\begin{proof} By (ii) of Lemma~\ref{lem:criterion} we have 
$$\int_{t=0}^{+\infty}t\mu\{\f<V_\theta-t\}dt\le A\int_{t=0}^{+\infty}t\ca\{\f<V_\theta-t\}dt\le C_1$$ 
uniformly for $\f\in\cE_C$, and the result follows. 
\end{proof} 

\begin{thm}\label{thm:cont}
 Let $\f\in\cE^1(X,\theta)$ and set $\mu:=\MA(\f)$. 
 Then $L_\mu$ is continuous on $\cE_C$ for each $C>0$ and $F_\mu=E-L_\mu$ is usc on $\cE^1(X,\theta)$. 
\end{thm}
\begin{proof} The second statement follows from the first. Indeed for each $A$ $\{F_\mu\ge A\}$ is contained in $\cE_C$ for some $C$ by Corollary~\ref{cor:proper}, and we conclude that $\{F_\mu\ge A\}$ is closed as desired if we know that $F_\mu$ is usc on $\cE_C$. 

In order to prove the first assertion, assume first that $\f$ has minimal singularities. Then the result follows from Corollary~\ref{cor:domcap}, since we have $\MA(\p)\le A\ca$ for some $A>0$. Indeed pick $t\ge 1$ such that $\p\ge V_\theta-t$. Then $t^{-1}\f+(1-t^{-1})V_\theta$ is a candidate in the definition of $\ca$, and the claim follows since 
$$\MA(\f)\le t^n\,\MA(t^{-1}\f+(1-t^{-1})V_\theta).$$ 
In the general case we write $\f$ as the decreasing limit of its canonical approximants $\f_k:=\max(\f,V_\theta-k)$. By Proposition~\ref{prop:cont} we have $I(\f_k,\f)\to 0$ as $k\to\infty$ and thus Lemma~\ref{lem:uniform} below yields that $L_{\MA(\f_k)}$ converges to $L_\mu$ uniformly on $\cE_C$. The result follows since for each $k$ $L_{\MA(\f_k)}$ is continuous on $\cE_C$ by the first part of the proof.
\end{proof}

\begin{lem}\label{lem:uniform} We have 
$$\sup_{\cE_C}\left|L_{\MA(\p_1)}-L_{\MA(\p_2)}\right|=O\left(I(\p_1,\p_2)^{1/2}\right),$$
uniformly for $\p_1,\p_2\in\cE_C$. 
\end{lem}
\begin{proof} Pick $\f\in\cE_C$ and set
$$a_p:=\int_X(\f-V_\theta)\MA(\p_1^{(p)},\p_2^{(n-p)}).$$
Our goal is to find $C_1>0$ only depending on $C$ and a bound on $|E(\p_1)|,|E(\p_2)|$ such that
$$|a_n-a_0|\le C_1 I(\p_1,\p_2)^{1/2}.$$ 
It is enough to consider the case where $\f,\p_1,\p_2$ furthermore have minimal singularities. Indeed in the general case one can apply the result to the canonical approximants with minimal singularities, and we conclude by continuity of mixed Monge-Amp{\`e}re operators along monotonic sequences. By integration by parts (\cite{BEGZ08} Theorem 1.14) we have
$$a_{p+1}-a_p=\int_\Omega(\f-V_\theta)dd^c(\p_1-\p_2)\wedge(\theta+dd^c\p_1)^p\wedge(\theta+dd^c\p_2)^{n-p-1}$$
$$=-\int_\Omega d(\f-V_\theta)\wedge d^c(\p_1-\p_2)\wedge(\theta+dd^c\p_1)^p\wedge(\theta+dd^c\p_2)^{n-p-1}$$
and the Cauchy-Schwarz inequality yields 
$$|a_{p+1}-a_p|^2\le A_p B_p$$
with 
$$A_p:=\int_\Omega  d(\f-V_\theta)\wedge d^c(\f-V_\theta)\wedge(\theta+dd^c\p_1)^p\wedge(\theta+dd^c\p_2)^{n-p-1}$$
and 
$$B_p:=\int_\Omega  d(\p_1-\p_2)\wedge d^c(\p_1-\p_2)\wedge(\theta+dd^c\p_1)^p\wedge(\theta+dd^c\p_2)^{n-p-1}\le I(\p_1,\p_2)$$
by (\ref{equ:I}). By integration by parts again we get
$$A_p=-\int_\Omega(\f-V_\theta)dd^c(\f-V_\theta)\wedge(\theta+dd^c\p_1)^p\wedge(\theta+dd^c\p_2)^{n-p-1}$$
$$=\int_\Omega(\f-V_\theta)\MA(V_\theta,\p_1^{(p)},\p_2^{(n-p-1)})-\int_\Omega(\f-V_\theta)\MA(\f,\p_1^{(p)},\p_2^{(n-p-1)})$$
which is uniformly bounded in terms of $C$ only by Lemma~\ref{lem:mixte}. We thus conclude that
$$|a_n-a_0|\le|a_n-a_{n-1}|+...+|a_1-a_0|\le C_1I(\p_1,\p_2)^{1/2}$$
for some $C_1>0$ only depending on $C$ as desired. 
\end{proof}

\section{Variational solutions of Monge-Amp{\`e}re equations}

\subsection{Variational formulation}
In this section we prove the following key step in our approach, which extends Theorem A of the introduction to the case of a big class. Recall that we have normalized the big cohomology class $\{\theta\}$ by requiring that $\vol(\theta)=1$. We let $\cM_X$ denote the set of all probability measures on $X$. For any $\mu\in\cM_X$ $E-L_\mu$ descends to a  concave functional
$$F_\mu:\cT^1(X,\theta)\to[-\infty,+\infty[.$$

\begin{thm}\label{thm:var} Given $T\in\cT^1(X,\theta)$ and $\mu\in\cM_X$ we have 
$$F_\mu(T)=\sup_{\cT^1(X,\theta)}F_\mu
\hskip.5cm \text{ iff } \hskip.5cm 
\mu=\langle T^n\rangle.
$$
\end{thm}

\begin{proof} 
Write $T=\theta+dd^c\f$ and suppose that $\mu=\langle T^n \rangle =\MA(\f)$. Since $E$ is concave we have
$$
E(\f)+\int_X(\p-V_\theta)\MA(\f)\ge E(\p)+\int_X(\f-V_\theta)\MA(\f).
$$
Indeed the inequality holds when $\f,\p$ have minimal singularities by (\ref{equ:eprime}) and the general case follows by approximating $\f$ by $\min(\f,V_\theta-j)$ and similarly for $\p$. It follows that 
$$F_\mu(T)=\sup_{\cT^1(X,\theta)}F_\mu.$$
In order to prove the converse we will rely on the differentiability result obtained by the first two authors (\cite{BB08} Theorem B). Given a usc function $u:X\to[-\infty,+\infty[$ we define its $\theta$-psh envelope by
$$P(u)=\sup\{\f\in\cP(X,\theta),\,\f\le u\text{ on } X\}$$
(or as $P(u):\equiv-\infty$ is the set of $\theta$-psh functions on the right is empty). Note that $P(u)$ is automatically usc. Indeed its usc majorant $P(u)^*\ge P(u)$ is $\theta$-psh and satisfies $P(u)^*\le u$ since $u$ is usc, and it follows that $P(u)=P(u)^*$ by definition. Note also that
$$V_\theta=P(0).$$
Now let $v$ be a continuous function on $X$. Since $v$ is in particular bounded, we see that $P(\f+tv)\ge\f-O(1)$ belongs to $\cE^1(X,\theta)$ for every $t\in\R$. We claim that the function 
$$g(t):=E(P(\f+tv))-L_\mu(\f)-t\int_Xv d\mu$$ achieves its maximum at $t=0$. Indeed since $P(\f+tv)\le\f+tv$ we have
$$g(t)\le E(P(\f+tv))- L_\mu(P(\f+tv))\le E-L_\mu(\f)=g(0)$$
by assumption since $P(\f+tu)\in\cE^1(X,\theta)$. By Lemma~\ref{lem:key} below it follows that 
$$0=g'(0)=\int_X v\,\MA(\f)-\int_Xv d\mu.$$
\end{proof}

\begin{lem}\label{lem:key} Given $\f\in\cE^1(X,\theta)$ and a continuous function $v$ on $X$ we have
$$\frac{d}{dt}_{t=0} E(P(\f+tv))=\int_X v\,\MA(\f).$$
\end{lem}
\begin{proof} By dominated convergence we get the following equivalent integral formulation
\begin{equation}\label{equ:integral} E(P(\f+v))-E(\f)=\int_{t=0}^1\langle v,\MA(P(\f+tv))\rangle dt.
\end{equation}
Since $\f$ is usc, we can write it as the decreasing limit of a sequence of continuous functions $u_j$ on $X$. It is then straightforward to check that for each $t\in\R$ $P(\f+tv)$ is the decreasing limit of $P(u_j+tv)$. By Theorem B of~\cite{BB08} we have
$$E(P(u_j+v))-E(P(u_j))=\int_{t=0}^1\langle v,\MA(P(u_j+tv))\rangle dt$$
for each $j$. By Proposition~\ref{prop:energy_gen} the energy $E$ is continuous along decreasing sequences hence
$$E(P(\f+tv))=\lim_{j\to\infty}E(P(u_j+tv))$$
and 
$$\langle v,\MA(P(\f+tv))\rangle=\lim_{j\to\infty}\langle v,\MA(P(u_j+tv))\rangle$$
by~\cite{BEGZ08} Theorem 1.17 since $P(\f+tv)$ has full Monge-Amp{\`e}re mass. We thus obtain
 (\ref{equ:integral}) by dominated convergence, since the total mass of $\MA(P(u_j+tv))$ is equal to $1$ for each $j$ and $t$.
\end{proof}

We introduce the \emph{Legendre transform} of $E$:
 \begin{defi} The \emph{electrostatic energy} of a probability measure $\mu$ on $X$ is defined as the Legendre transform
$$E^*(\mu):=\sup_{\cT^1(X,\theta)}F_\mu.$$
We will say that $\mu$ has \emph{finite energy} if $E^*(\mu)<+\infty$. 
\end{defi}
Note that $E^*(\mu)\ge 0$ since $E(V_\theta)=L_\mu(V_\theta)=0$. We thus get a convex functional
$$E^*:\cM_X\to[0,+\infty],$$
which is furthermore lower semi-continuous (in the weak topology of measures) by Lemma~\ref{lem:usc}. 

Here is a first characterization of measures $\mu$ with finite energy.  
\begin{lem}\label{lem:finite} A probability measure $\mu$ has finite energy iff $L_\mu$ is finite on $\cE^1(X,\theta)$. In that case $\mu$ is necessarily non-pluripolar. 
\end{lem}
\begin{proof} If $L_\mu$ is finite on $\cE^1(X,\theta)$ then $F_\mu:=E-L_\mu$ is $J$-proper on $\cT^1(X,\theta)$ and is bounded on each $J$-sublevel set by Corollary~\ref{cor:proper}, and the result follows. 
\end{proof}

The next result shows that $E$ is in turn the Legendre transform of $E^*$.
\begin{prop}\label{prop:legendre} For any $\f\in\cE^1(X,\theta)$ we have
$$E(\f)=\inf_{\mu\in\cM_X}\left(E^*(\mu)+L_\mu(\f)\right).$$
\end{prop}
\begin{proof} We have $E^*(\mu)\ge E(\f)-L_\mu(\f)$ and equality holds for $\mu=\MA(\f)$ by Theorem~\ref{thm:var}. The result follows immediately.
\end{proof}
We can alternatively relate $E^*$ and $J$ as follows. If $\mu$ is a probability measure on $X$ we define an affine functional $H_\mu$ on $\cT(X,\theta)$ by setting
$$H_\mu(T):=\int(\f-V_\theta)\left(\MA(V_\theta)-\mu\right)$$
with $T=\theta+dd^c\f$. Then we have 
$$E^*(\mu)=\sup_{T\in\cT^1(X,\omega)}\left(H_\mu(T)-J(T)\right),$$
and Theorem~\ref{thm:var} combined with the uniqueness result of~\cite{BEGZ08} says that the supremum is attained (exactly) at $T$ iff $\mu=\langle T^n\rangle$. 
 
\subsection{Direct method} 

We will also use the following technical result.
\begin{lem}\label{lem:bound} Let $\nu$ be a measure with finite energy and let $A>0$. Then $E^*$ is bounded on 
$$\{\mu\in\cM_X|\mu\le A\nu\}.$$
\end{lem}
\begin{proof} By Proposition~\ref{prop:square} there exists $B>0$ such that
$$\sup_{\cE_C}|I_\nu|\le B(1+C^{1/2})$$
for all $C>0$, hence 
$$\sup_{\cE_C}|L_\mu|\le AB(1+C^{1/2})$$
for all $\mu\in\cM_X$ such that $\mu\le A\nu$. It follows that 
$$E^*(\mu)=\sup_{\cE^1(X,\theta)}(E-L_\mu)$$
$$\le\sup_{C>0}\left(AB(1+C^{1/2})-C\right)<+\infty.$$
\end{proof}

We are now in a position to state one of our main results
(see Theorem A of the introduction).

\begin{thm}\label{thm:main} A probability measure $\mu$ on $X$ has finite energy iff there exists $T\in\cT^1(X,\theta)$ such that $\mu=\langle T^n\rangle$. In that case $T=T_\mu$ is unique and satisfies 
$$n^{-1}E^*(\mu)\le J(T_\mu)\le n E^*(\mu).$$
Furthermore any maximizing sequence $T_j\in\cT^1(X,\theta)$ for $F_\mu$ converges to $T_\mu$. 
\end{thm}
\begin{proof} Suppose first that $\mu=\langle T^n\rangle$ for some $T\in\cE^1(X,\theta)$. Then $\mu$ has finite energy by Lemma~\ref{lem:finite}. Uniqueness follows from~\cite{BEGZ08}, where it was more generally proved that a current $T\in\cT(X,\theta)$ with full Monge-Amp{\`e}re mass is determined by $\langle T^n\rangle$ by adapting Dinew's proof~\cite{Din09} in the K{\"a}hler case. 

Write $T=\theta+dd^c\f$. By the easy part of Theorem~\ref{thm:var} we have
$$E^*(\mu)=E(\f)-\int_X(\f-V_\theta)\MA(\f)=J_\f(V_\theta)$$
and the second assertion follows from Lemma~\ref{lem:quasisym}. 

Now let $T_j\in\cT^1(X,\theta)$ be a maximizing sequence for $F_\mu$. Since $F_\mu$ is $J$-proper the $T_j$'s stay in a compact set, so we may assume that they converge towards $S\in\cT^1(X,\theta)$ and we are to show that $S=T$. Now $F_\mu$ is usc by Theorem~\ref{thm:cont} thus $F_\mu(S)$ has to be equal to $\sup_{\cT^1(X,\theta)}F_\mu$. By Theorem~\ref{thm:var} we thus get 
$$\langle S^n\rangle=\mu=\langle T^n\rangle$$
hence $S=T$ as desired by uniqueness. 

We now come to the main point. Assume that $\mu$ has finite energy in the above sense that $E^*(\mu)<+\infty$. In order to find $T\in\cT^1(X,\theta)$ such that $\langle T^n\rangle=\mu$ it is enough to show by Theorem~\ref{thm:var} that $F_\mu$ achieves its supremum on $\cT^1(X,\theta)$. Since $F_\mu$ is $J$-proper it is even enough to that $F_\mu$ is usc, which we know holds true \emph{a posteriori} by Theorem~\ref{thm:cont}. 

We are unfortunately unable to establish this \emph{ a priori}, thus we resort to a more indirect argument. Assume first that $\mu\le A\ca$ for some $A>0$. Corollary~\ref{cor:domcap} then implies that $L_\mu$ is continuous on $\cE_C$ for each $C$, hence $F_\mu$ is usc in that case and we infer that $\mu=\langle T^n\rangle$ for some $T\in\cT^1(X,\theta)$ as desired. 

In the general case we rely on the following result already used in~\cite{GZ07,BEGZ08} and which basically goes back to Cegrell~\cite{Ceg98}. 

\begin{lem}\label{lem:rn} Let $\mu$ be a probability measure that puts no mass on pluripolar subsets. Then $\mu$ is absolutely continuous with respect to a probability measure $\nu$ such that $\nu\le\ca$. 
\end{lem}
\begin{proof} As in~\cite{Ceg98} we apply the generalised Radon-Nikodym theorem to the compact convex set of measures 
$$\cC:=\{\nu\in\cM_X, \nu\le\ca\}.$$
By Proposition~\ref{prop:closed} this is indeed a closed subset of $\cM_X$ hence is compact. By~\cite{Rai69} there exists $\nu\in\cC$, $\nu'\perp\cC$ and $f\in L^1(\nu)$ such that
$$\mu=f\nu+\nu'.$$
Since $\mu$ puts no mass on pluripolar sets and $\cC$ characterises such sets, it follows that $\nu'=0$, qed.
\end{proof}

Since $\mu$ is non-pluripolar by Lemma~\ref{lem:finite}, we can use  Lemma~\ref{lem:rn} and write $\mu=f\nu$ with $\nu\le\ca$ and $f\in L^1(\nu)$. Now set
$$
\mu_k:=(1+\e_k)\min(f,k)\nu
$$
where $\e_k\ge 0$ is chosen so that $\mu_k$ has total mass $1$. We thus have $\mu_k\le 2k\ca$ thus by the above first part of the proof we have $\mu_k=\langle T_k^n\rangle$ for some $T_k\in\cT^1(X,\theta)$. On the other hand we have $\mu_k\le 2\mu$ for all $k$ thus $E^*(\mu_k)$ is uniformly bounded by Lemma~\ref{lem:bound}. By the first part of the proof it follows that all $T_k$ stay in a sublevel set $\{J\le C\}$. 
Since the latter is compact we may assume that $T_k\to T$ for some $T\in\cT^1(X,\theta)$. In particular $T$ has full Monge-Amp{\`e}re mass and~\cite{BEGZ08} Corollary 2.21 thus yields
$$\langle T^n\rangle\ge(\liminf_{k\to\infty}(1+\e_k)\min(f,k))\nu=\mu,$$
hence $\langle T^n\rangle=\mu$ since both measures have total mass 1, qed.  
\end{proof}

By a similar argument we can now recover the main result of~\cite{BEGZ08}. 

\begin{cor} \label{cor:begz}
Let $\mu$ be a non-pluripolar probability measure on $X$. Then there exists $T\in\cT(X,\theta)$ such that $\mu=\langle T^n\rangle$. 
\end{cor} 

\begin{proof} Using Lemma~\ref{lem:rn} as above we can write $\mu=f\nu$ with $\nu\le\ca$ and $f\in L^1(\nu)$, and we set $\mu_k=(1+\e_k)\min(f,k)\nu$ as above. By Theorem~\ref{thm:main} there exists $T_k\in\cT^1(X,\theta)$ such that $\mu_k=\langle T_k^n\rangle$.  We may assume that $T_k$ converges to some $T\in\cT(X,\theta)$.

We claim that $T$ has full Monge-Amp{\`e}re mass, which will imply $\langle T^n\rangle=\mu$ by~\cite{BEGZ08} Corollary 2.21 just as above. Write $T=\theta+dd^c\f$ and $T_k=\theta+dd^c\f_k$ with $\sup_X\f=\sup_X\f_k=0$ for all $k$. By general Orlicz space theory (\cite{BEGZ08} Lemma 3.3) there exists a convex non-decreasing function $\chi:\R_-\to\R_-$ with a sufficiently slow growth at $-\infty$ and $C>0$ such that
$$\int_X(-\chi)(\p-V_\theta)d\mu\le\int_X(\p-V_\theta)d\nu+C$$
for all $\p\in\cP(X,\theta)$ normalized by $\sup_X\p=0$. Now $\int_X(\f_k-V_\theta)d\nu=L_\mu(\f_k)$ is uniformly bounded by Corollary~\ref{cor:domcap}, and we infer that
$$\int_X(-\chi)(\f_k-V_\theta)\MA(\f_k)\le 2\int_X(-\chi)(\f_k-V_\theta)d\mu$$
is uniformly bounded. This means that the $\chi$-weighted energy (cf.~\cite{BEGZ08}) of $\f_k$ is uniformly bounded (since $\f_k$ has full Monge-Amp{\`e}re mass) and we conclude that $\f$ has finite $\chi$-energy by semi-continuity of the $\chi$-energy. This implies in turn that $\f$ has full Monge-Amp{\`e}re as desired. 
\end{proof}

\section{\label{sec:pluri}Pluricomplex electrostastics}
We assume throughout this section that $\theta=\omega$ is a K{\"a}hler form (still normalized by $\int_X\omega^n=1$). We then have $V_\omega=0$. 
 
\subsection{Pluricomplex energy of measures}  
We first record the following useful explicit formulas. 
\begin{lem}\label{lem:formules} Let $\mu$ be a probability measure with finite energy, and write $\mu=(\omega+dd^c\f)^n$ with $\f\in\cE^1(X,\omega)$. Then we have
\begin{equation}\label{equ:energymes}
E^*(\mu)=\frac{1}{n+1}\sum_{j=0}^{n-1}\int_X\f\left((\omega+dd^c\f)^j\wedge\omega^{n-j}-\mu\right)$$
$$=\sum_{j=0}^{n-1}\frac{j+1}{n+1}\int_X d\f\wedge d^c\f\wedge (\omega+dd^c\f)^j\wedge\omega^{n-j}.
\end{equation}
\end{lem}
\begin{proof} By the easy part of Theorem~\ref{thm:var} we have 
$$E^*(\mu)=E(\f)-\int_X\f d\mu=J_\f(0)$$
and the formulas follow from the explicit formulas for $E$ and $J_\f(\p)$ given in Section~\ref{sec:energy}. 
\end{proof}

When $X$ is a compact Riemann surface ($n=1$) a given probability measure $\mu$ may be written $\mu=\omega+dd^c\f$ by solving Laplace's equation. Then $E^*(\mu)<+\infty$ iff $\f$ belongs to the Sobolev space $L^2_1(X)$, and in that case 
$$2E^*(\mu)=\int_X\f(\omega-\mu)=\int_X d\f\wedge d^c\f$$
is nothing but the classical Dirichlet functional applied to the potential $\f$. 

\smallskip

We now indicate the relation with the classical \emph{logarithmic energy} (cf.~\cite{ST}  Chapter 1). Recall that a signed measure $\lambda$ on $\C$ is said to have finite logarithmic energy if $(z,w)\mapsto\log|z-w|$ belongs to $L^1(|\lambda|\otimes|\lambda|)$, and its logarithmic energy is then defined by
$$I(\lambda)=\int\int\log|z-w|^{-1}\lambda(dz)\lambda(dw).$$
When $\lambda$ has finite energy its \emph{logarithmic potential} 
$$U_\lambda(z)=\int\log|z-w|\lambda(dw)$$
belongs to $L^1(|\lambda|)$ and we have
$$I(\lambda)=-\int U_\lambda(z)\lambda(dz).$$
The Fubiny-Study form $\omega$ (normalized to mass $1$) has finite energy and a simple computation in polar coordinates yields $I(\omega)=-1/2$. We also have
$$U_\omega(z)=\frac{1}{2}\log(1+|z|^2).$$
The logarithmic energy $I$ can be polarized into a quadratic form
$$I(\lambda,\mu):= \int\int\log|z-w|^{-1}\lambda(dz)\mu(dw)$$
on the vector space of signed measures with finite energy, which then splits into the $I$-orthogonal sum of $\R\omega$ and of the space of signed measures with total mass $0$. The quadratic form $I$ is positive definite on the latter space (\cite{ST} Lemma I.1.8). 

\begin{lem}\label{lem:p1} 
Let $X=\PP^1$ and $\omega$ to be the Fubini-Study form normalized to mass $1$. If $\mu$ is a probability measure on $\C\subset\PP^1$ then $E^*(\mu)<+\infty$ iff $\mu$ has finite logarithmic energy and in that case we have
$$E^*(\mu)=\frac{1}{2}I(\mu-\omega).$$
\end{lem}

\begin{proof} 
We have $\mu=\omega+dd^c(U_\mu-U_\omega)$, so the first assertion means that $\mu$ has finite logarithmic energy iff $U_\mu-U_\omega$ belongs to the Sobolev space $L^2_1(\PP^1)$, which is a classical fact. The second assertion follows from (\ref{equ:energymes}), which yields
$$2E^*(\mu)=-\int(U_\mu-U_\omega)(\mu-\omega)=I(\mu-\omega).$$
\end{proof}

\subsection{A pluricomplex electrostatic capacity}
As in~\cite{BB08} we consider a \emph{weighted subset} consisting of a compact subset $K$ of $X$ together with a continuous function $v\in C^0(K)$, and we define the \emph{equilibrium weight} of $(K,v)$ as the extremal function
 
 $$
 P_Kv:=\text{sup}^*\{\f|\f\in\cP(X,\omega),\,\f\le v\text{ on }K\}.
 $$

The function $P_Kv$ belongs to $\cP(X,\omega)$ if $K$ is non-pluripolar and satisfies $P_Kv\equiv+\infty$ otherwise (cf.~\cite{Sic81}, \cite{GZ05}). 

If $K$ is a compact subset of $\C^n$ and 
$$\f_{FS}:=\frac{1}{2}\log (1 + \vert z\vert^2)$$
denotes the potential on $\C^n$ of the Fubiny-Study metric, then $P_K(-\f_{FS})+\f_{FS}$ coincides with Siciak's extremal function, i.e. the usc upper envelope of the family of all psh functions $u$ on $\C^n$ with logarithmic growth such that $u\le 0$ on $K$. 

The \emph{equilibrium measure} of a non-pluripolar weighted compact set $(K,v)$ is defined as 

 $$
 \eq(K,v):=\MA(P_K v)
 $$

 and its \emph{energy at equilibrium} is 
 $$
 \eneq(K,v):=E(P_K v).
 $$ 

The functional $v\mapsto\eneq(K,v)$ is concave and G{\^a}teaux differentiable on $C^0(K)$, with directional derivative at $v$ given by integration against $\eq(K,v)$ by Theorem B of~\cite{BB08}. As a consequence of Theorem~\ref{thm:var} we get the following related variational characterization of $\eq(K,v)$. 

Let $\cM_K$ denote the set of all probability measures on $K$.  

 \begin{thm}\label{thm:equi} 
 If $(K,v)$ is a non-pluripolar weighted compact subset then we have
 $$
 \eneq(K,v)=\inf_{\mu\in\cM_K}\left(E^*(\mu)+\langle v,\mu\rangle\right)
 $$
 and the infimum is achieved precisely for $\mu=\eq(K,v)$. 
 
Conversely if $K$ is pluripolar then $E^*(\mu)=+\infty$ for each $\mu\in\cM_K$. 
\end{thm}

\begin{proof} 
Assume first that $K$ is non-pluripolar. The concave functional $F:=\eneq(K,\cdot)$ is non-decreasing on $C^0(K)$ and satisfies the scaling property $F(v+c)=F(v)+c$ so its Legendre transform 
$$F^*(\mu):=\sup_{v\in C^0(K)}\left(F(v)-\langle v,\mu\rangle\right)$$
is necessarily infinite outside $\cM_K\subset C^0(K)^*$. The basic theory of convex functions thus yields 
$$F(v)=\inf_{\mu\in\cM_K}\left(F^*(\mu)+\langle v,\mu\rangle\right)$$
and the infimum is achieved exactly at $\mu=F'(v)=\eq(K,v)$. What we have to show is thus
$F^*=E^*|_{\cM_K}$. But on the one hand $P_K(v)\le v$ on $K$ implies
$$F^*(\mu)\le\sup_{v\in C^0(K)}\left( E(P_Kv)-\langle P_Kv,\mu\rangle\right)$$
$$\le\sup_{\f\in\cE^1(X,\omega)}\left(E(\f)-\langle\f,\mu\rangle\right)=E^*(\mu).$$
On the other hand every $\f\in\cE^1(X,\omega)$ has identically zero Lelong numbers, so it can be written as a decreasing limit of \emph{smooth} $\omega$-psh functions $\f_j$ by~\cite{Dem92}. For each $j$ the function $v_j:=\f_j|_K\in C^0(K)$ satisfies $\f_j\le P_K(v_j)$ hence 
$$E(\f_j)-\langle\f_j,\mu\rangle\le E(P_K v_j)-\langle v_j,\mu\rangle\le F^*(\mu)$$
and we infer $E^*(\mu)\le F^*(\mu)$ as desired since 
$$E(\f)-\langle\f,\mu\rangle=\lim_{j\to\infty}\left(E(\f_j)-\langle\f_j,\mu\rangle\right)$$
by Proposition~\ref{prop:energy_min} and monotone convergence respectively. 
 
Now assume that $K$ is pluripolar. If there exists $\mu\in\cM_K$ with $E^*(\mu)<+\infty$. then Theorem A implies in particular that $\mu$ puts no mass on pluripolar sets, which contradicts $\mu(K)=1$. 
\end{proof}

 \smallskip

 One can interpret Theorem~\ref{thm:equi} as a pluricomplex version of weighted electrostatics where $K$ is  a condenser, $\mu$ describes  a charge distribution on $K$, $E^*(\mu)$ is its internal electrostatic energy and $\langle v,\mu\rangle$ is the external energy induced by the  
 field $v$. The equilibrium distribution $\eq(K,v)$ is then the unique minimizer of the total energy $E^*(\mu)+\langle v,\mu\rangle$ of the system. 

In view of Theorem~\ref{thm:equi} it is natural to define the \emph{electrostatic capacity} of a weighted compact subset $(K,v)$ by
$$-\log C_e(K,v)=\frac{n+1}{n}\inf\{E^*(\mu)+\langle v,\mu\rangle,\,\mu\in\cM_K\}.$$
We then have $C_e(K,v)=0$ iff $K$ is pluripolar, and 
$$C_e(K,v)=\exp\left(-\frac{n+1}{n}\eneq(K,v)\right)$$
when $K$ is non-pluripolar. 

Our choice of constants is guided by~\cite{BB08} Corollary A, which shows that $C_e(K,v)$ coincides (up to a multiplicative constant) with the natural generalization of Leja-Zaharjuta's \emph{transfinite diameter} when $\omega$ is the curvature form of a metric on ample line bundle $L$ over $X$. In particular this result shows that the Leja-Zaharjuta transfinite diameter $d_\infty(K)$ of a compact subset $K\subset\C^n$, normalized so that 
$$d_\infty(tK)=t d_\infty(K)$$ 
for each $t>0$, is proportional to $C_e(K,-\f_{FS})$.

By the continuity properties of extremal functions and of the energy functional along monotone sequences, it follows that the capacity $C_{e}(\cdot,v)$ can be extended in the usual way as an outer Choquet capacity on $X$ which vanishes exactly on pluripolar sets. In view of Lemma~\ref{lem:p1} this electrostatic capacity extends the classical logarithmic capacity of a compact subset $K\subset\C$, which is equal to
 $$\exp\left(-\inf\{I (\mu),\,\mu \in \mathcal M_K\}\right).
 $$
On the other hand the \emph{Alexander-Taylor capacity} of a weighted compact subset $(K,v)$ may be defined by 
$$
 T(K,v) :=\exp(-\sup_X P_Kv).
 $$ 
(compare~\cite{AT84,GZ05}). We thus have $T(K,v)=0$ iff $K$ is pluripolar.  We have for instance 
$$T(B_R,0)=\frac{R}{(1+R^2)^{1/2}}$$
when $X=\PP^n$ and $B_R\subset\C^n$ is the ball of radius $R$ (cf.~\cite{GZ05} Example 4.11). In particular this implies $T(B_R,v)\simeq R$ as $R\to 0$. 

The two capacities compare as follows.
\begin{prop} There exists $C>0$ such that 
$$
T(K,v)^{1+1/n}\le C_e(K,v)\le C e^{M}T(K,v)^{1/n}
$$
for each $M>0$ and each weighted compact subset $(K,v)$ so that $v\ge -M$ on $K$. 
\end{prop}

\begin{proof} 
The definition of $E$ immediately implies that 
$$\eneq(K,v)=E(P_Kv)\le\sup_X P_Kv$$ 
hence the left-hand inequality. Conversely $v\ge -M$ implies $P_Kv\ge -M$ hence Proposition~\ref{prop:energy_min} yields
$$\int_X (P_Kv)\omega^n-nM\le (n+1)\eneq(K,v).$$
But there exists a constant $C>0$ such that 
$$\sup_X\f\le\int_X\f\omega^n+C$$
for all $\f\in\cP(X,\omega)$ by compactness of $\cT(X,\omega)$, and we get
$$\frac{1}{n}\sup_X P_Kv\le\frac{n+1}{n}\eneq(K,v)+M+C'$$
as desired. 
\end{proof}

Observe that when $K$ lies in the unit ball of $\C^n\subset \PP^n$ and 
$$
v(z)= -\frac{1}{2}\log (1 + \vert z\vert^2)
$$ 
then we get $v\ge -\log\sqrt{2}$ on $K$ and the above results improve on~\cite{LT83}.

\section{\label{sec:KE}Variational principles for K{\"a}hler-Einstein metrics}

In this section we use the variational approach to study the existence of K{\"a}hler-Einstein metrics on manifolds  with definite
first Chern class. The Ricci-flat case is an easy consequence of Theorem A. In Section~\ref{sec:gentype} we treat the case of manifolds
of general type and prove Theorem C. The more delicate case of Fano manifolds occupies the remaining sections:
in Section~\ref{sec:cont} we construct continuous geodesics in the space of positive closed currents with precribed cohomology class,
we then prove Theorem D  in Section~\ref{sec:fano}, while uniqueness of (singular) K{\"a}hler-Einstein metrics with
positive curvature (Theorem E) is established in Section~\ref{sec:unique}. We will use throughout the convenient language of weights, i.e. view metrics additively. We refer for instance to~\cite{BB08} for explanations.

\subsection{\label{sec:gentype} Manifolds of general type}

Let $X$ be a smooth projective variety of \emph{general type}, i.e.~such that $K_X$ is big. A weight $\phi$ on $K_X$ induces a volume form $e^{2\phi}$. By a \emph{singular K{\"a}hler-Einstein weight} we mean a psh weight on $K_X$ such that $\MA(\phi)=e^{2\phi}$ \emph{and} such that $\int_X e^{2\phi}=\vol(K_X)=:V$, or equivalently such that $\MA(\phi)$ has \emph{full Monge-Amp{\`e}re mass}. 

In~\cite{EGZ09} a singular K{\"a}hler-Einstein weight was constructed using the existence of the \emph{canonical model} 
$$X_{\text{can}}:=\text{Proj}\oplus_{m\ge 0}H^0(X,mK_X)$$
provided by the fundamental result of~\cite{BCHM06}. In~\cite{Tsu06} a direct proof of the existence of a singular K{\"a}hler-Einstein weight was sketched and the argument was expanded in~\cite{ST08}. In~\cite{BEGZ08}  existence and uniqueness of singular K{\"a}hler-Einstein weights was established using a generalized comparison principle, and the unique singular K{\"a}hler-Einstein weight was furthermore shown to have \emph{minimal singularities} in the sense of Demailly.

\smallskip

We propose here to give a direct variational proof of the existence of a singular K{\"a}hler-Einstein weight in $\cE^1(K_X)$ (we therefore don't recover the full force of the result in~\cite{BEGZ08}). 
We proceed as before but replacing the functional $F_\mu$ by $F_+:=E-L_+$ where we have set
$$
L_+(\phi):=\frac{1}{2}\log\int_X e^{2\phi}.
$$

\noindent {\bf Proof of Theorem C.}
Note that $e^{2\phi}$ has $L^\infty$-density with respect to Lebesgue measure. Indeed if $\phi_0$ is a given smooth weight on $K_X$ we have $e^{2\phi}=e^{2\phi-2\phi_0}e^{2\phi_0}$ where $e^{2\phi_0}$ is a smooth positive volume form and the \emph{function} $\phi-\phi_0$ is bounded from above on $X$. Given $\phi_1,\phi_2\in\cP(K_X)$, we can in particular consider the integral
$$
\int_X(\phi_1-\phi_2)e^{2\phi}:=\int_X(\phi_1-\phi_0)e^{2\phi}-\int_X(\phi_2-\phi_0)e^{2\phi},
$$
which is of course independent of the choice of $\phi_0$. 

\begin{lem}\label{lem:direc} The directional derivatives of $L_+$ on $\cP(K_X)$ are given by
$$\frac{d}{dt}_{t=0_+}L_+(t\phi+(1-t)\psi)=\frac{\int_X(\phi-\psi)e^{2\psi}}{\int_X e^{2\psi}}.$$
\end{lem}
\begin{proof} By the chain rule it is enough to show that
$$\frac{d}{dt}_{t=0_+}\int_Xe^{t\phi+(1-t)\psi}=\int_X(\phi-\psi)e^{\psi}.$$
One has to be a little bit careful since $\phi-\psi$ is not bounded on $X$. But we have
$$\int_X\left(e^{t\phi+(1-t)\psi}-e^{\psi}\right)=\int_X\left(e^{t(\phi-\psi)}-1\right)e^\psi.$$
Now $(e^{t(\phi-\psi)}-1)/t$ decreases pointwise to $\phi-\psi$ as $t$ decreases to $0$ by convexity of $\exp$ and the result indeed follows by monotone convergence.
\end{proof}

Using this fact and arguing exactly as in Theorem~\ref{thm:var} proves that 
\begin{equation}\label{equ:max+}F_+(\phi)=\sup_{\cE^1(K_X)}F_+
\end{equation}
implies
\begin{equation}\label{equ:MA+}\MA(\phi)=e^{2\phi+c}
\end{equation}
for some $c\in\R$. Indeed apart from~\cite{BB08} the main point of the proof of Theorem~\ref{thm:var} is that $E(P(\phi+v))-L_\mu(\phi+v)$ is maximum for $v=0$ if $E-L_\mu$ is maximal at $\phi$, and this only relied on the fact that $L_\mu$ is non-decreasing, which is also the case for $L_+$. 

Conversely $E$ is concave while $L_+$ is convex by H{\"o}lder's inequality, thus $F_+$ is concave and (\ref{equ:MA+}) implies (\ref{equ:max+}) as in
Theorem~\ref{thm:var}. 

In order to conclude the proof of Theorem C we need to prove that $F_+$ achieves its supremum on $\cE^1(K_X)$, or equivalently on $\cT^1(K_X)$. Now Corollary~\ref{cor:proper} applies to $F_+=E-L_+$ since $L_+$ is non-decreasing, convex and satisfies the scaling property, and we conclude that $F_+$ is $J$-proper as before. It thus remains to check that $F_+$ is upper semicontinuous, which will follow if we prove that $L_+$ is usc on $\cE_C$ for each $C$ as before.

But we claim that $L_+$ is actually continuous on $\cP(K_X)$. Indeed let $\phi_j\to\phi$ be a convergent sequence in $\cP(K_X)$. Upon extracting we may assume that $\phi_j\to\phi$ a.e. On the other hand, given a reference weight $\phi_0$, $\sup_X(\phi_j-\phi_0)$ is uniformly bounded by Hartogs' lemma, thus $e^{2(\phi_j-\phi_0)}$ is uniformly bounded and we get $\int_X e^{2\phi_j}\to\int_Xe^{2\phi}$ as desired by dominated convergence.

\subsection{\label{sec:cont} Continuous geodesics}
Let $\omega$ be a semi-positive $(1,1)$-form on $X$. If $Y$ is a complex manifold, then a map $\Phi:Y\to\cP(X,\omega)$ will be said to be psh (resp.~locally bounded, continuous, smooth) iff the induced function $\Phi(x,y):=\Phi(y)(x)$ on $X\times Y$ is $\pi_X^*\omega$-psh (resp.~locally bounded, continuous, smooth). We shall also say that $\Phi$ is \emph{maximal} if it is psh, locally bounded and 
$$(\pi_X^*\omega+dd^c_{(x,y)}\Phi)^{n+m}=0$$
where $m:=\dim Y$ and $dd^c_{(x,y)}$ acts on both variables $(x,y)$. If $Y$ is a radially symmetric domain in $\C$ and $\Phi$ is smooth on $X\times\overline{Y}$ such that $\omega+dd^c_x\Phi(\cdot,y)>0$ for each $y\in Y$ then 
by definition $\Phi$ is flat iff $\Phi(e^t)$ is a \emph{geodesic}  for the Riemannian metric on
$$
\{\f\in C^\infty(X),\omega+dd^c\f>0\}
$$
defined in~\cite{Mab87,Sem92,Don99}. 

\begin{prop}\label{prop:flat} 
If $\Phi:Y\to\cP(X,\omega)$ is a psh map then $E\circ\Phi$ is a psh function on $Y$ (or is indentically $-\infty$ on some component of $Y$). When $\Phi$ is furthermore locally bounded we have
\begin{equation}\label{equ:ddcE}
dd^c_y(E\circ\Phi)=(\pi_Y)_*\left((\pi_X^*\omega+dd^c_{(x,y)}\Phi)^{n+1}\right).
\end{equation}
In particular if $\dim Y=1$ then $E\circ\Phi$ is harmonic on $Y$ if $\Phi$ is maximal (=harmonic in this case). 
\end{prop}

\begin{proof} 
Assume first that $\Phi$ is smooth. Then  we can consider
\begin{equation}\label{equ:ecirc}E\circ\Phi:=\frac{1}{n+1}(\pi_Y)_*\left(\Phi\sum_{j=0}^n(\pi_X^*\omega+dd^c_x\Phi)^j\wedge\pi_X^*\omega^{n-j}\right).\end{equation}
The formula
$$dd^c_y(E\circ\Phi)=(\pi_Y)_*\left((\pi_X^*\omega+dd^c_{(x,y)}\Phi)^{n+1}\right)$$
follows from an easy but tedious computation relying on integration by parts and will be left to the reader. 

When $\Phi(x,y)$ is bounded and $\pi_X^*\omega$-psh the same argument works. Indeed integration by parts is a consequence of Stokes formula applied to a \emph{local} relation of the form $u=dv$, and the corresponding relation in the smooth case can be extended to the bounded case by a local regularization argument. 

Finally let $\Phi(x,y)$ be an arbitrary $\pi_X^*\omega$-psh function. We may then write $\Phi$ as the decreasing limit of $\max(\Phi,-k)$ as $k\to\infty$, and by Proposition~\ref{prop:energy_gen} $E\circ\Phi$ is then the pointwise decreasing limit of $E\circ\Phi_k$, whereas 
$$(
\pi_X^*\omega+dd^c_{(x,y)}\Phi_k)^{n+1}\to (\pi_X^*\omega+dd^c_{(x,y)}\Phi)^{n+1}
$$
by Bedford-Taylor's monotonic continuity theorem.
\end{proof} 

\begin{prop}\label{prop:extend} Let $\Omega\Subset\C^m$ be a smooth strictly pseudoconvex domain and let $\f:\partial\Omega\to\cP(X,\omega)$ be a continuous map. Then there exists a unique continuous extension $\Phi:\overline\Omega\to\cP(X,\omega)$ of $\f$ which is maximal on $\Omega$. 
\end{prop}

The proof is a simple adaptation of Bedford-Taylor's techniques to the present situation.
Although it has recently appeared in \cite{BD09} we include a proof as a courtesy to the reader.

\begin{proof} 
Uniqueness follows from the maximum principle. Let $\cF$ be the set of all continuous psh maps $\Psi:\overline\Omega\to\cP(X,\omega)$ such that $\Psi\le\f$ on $\partial\Omega$. Note that $\cF$ is non-empty since it contains all sufficiently negative constant functions of $(x,y)$. Let $\Phi$ be the upper envelope of $\cF$. We are going to show that $\Phi=\f$ on $\partial\Omega$ and that $\Phi$ is continuous. The latter property will imply that $\Phi$ is $\pi_X^*\omega$-psh, and it is then standard to show that $\Phi$ is maximal on $\Omega$ by using local solutions to the homogeneous Monge-Amp{\`e}re equation (compare.~\cite{Demsurv} P.17, \cite{BB08} Proposition 1.10). 

Assume first that $\f$ is a smooth. We claim that $\f$ admits a smooth psh extension $\widetilde{\f}:\overline\Omega\to\cP(X,\omega)$. Indeed we first cover $\overline\Omega$ by two open subsets $U_1,U_2$ such that $U_1$ retracts smoothly to $\partial\Omega$. We can then then extend $\f$ to a smooth map $\f_1:U_1\to\cP(X,\omega)$ using the retraction and pick any constant map $\f_2:U_2\to\cP(X,\omega)$. Since $\cP(X,\omega)$ is convex $\theta_1\f_1+\theta_2\f_2$ defines a smooth extension $\overline\Omega\to\cP(X,\omega)$ (where $\theta_1,\theta_2$ is a partition of unity adapted to $U_1,U_2$). Now let $\chi$ be a smooth strictly psh function on $\overline\Omega$ vanishing on the boundary of $\Omega$. Then $\widetilde{\f}:=\theta_1\f_1+\theta_2\f_2+C\chi$ yields the desired smooth psh extension of $\f$. 

Since $\widetilde{\f}$ belongs to $\cF$ we get in particular $\widetilde{\f}\le\Phi$ hence $\Phi=\f$ on $\partial\Omega$. We now take care of the continuity of $\Phi$, basically following~\cite{Demsurv} P.13. By~\cite{Dem92} the exists a sequence $\Phi_k$ of smooth functions on $X\times\overline\Omega$ which decrease pointwise to the usc regularization $\Phi^*$ and such that 
$$dd^c\Phi_k\ge-\e_k(\pi_X^*\omega+dd^c\chi).$$
Note that $\Psi_k:=(1-\e_k)(\Phi_k+\e_k\chi)$ is thus $\pi_X^*\omega$-psh. Given $\e>0$ we have $\Phi^*<\widetilde{\f}+\e$ on a compact neighbourhood $U$ of $X\times\partial\Omega$ thus $\Psi_k<\widetilde{\f}+\e$ on $U$ for $k\gg 1$. It follows that $\max(\Psi_k-\e,\widetilde{\f})$ belongs to $\cF$, so that $\Psi_k-\e\le\Phi$, and we get
$$\Phi\le\Phi^*\le\Phi_k\le(1-\e_k)^{-1}(\Phi+\e)-\e_k\chi,$$
which in turn implies that $\Phi_k$ converges to $\Phi$ uniformly on $X\times\overline\Omega$. We conclude that $\Phi$ is continuous in that case as desired. 

Let now $\f:\partial\Omega\to\cP(X,\omega)$ be an arbitrary continuous map. By Richberg's approximation theorem (cf.~e.g.~\cite{Dem92}) we may find a sequence of smooth functions $\f_k:\partial\Omega\to\cP(X,\omega)$ such that $\sup_{X\times\partial\overline\Omega}|\f-\f_k|=:\e_k$ tends to $0$. The corresponding envelopes $\Phi_k$ then satisfy $\Phi_k-\e_k\le\Phi\le\Phi_k+\e_k$, which shows that $\Phi_k\to\Phi$ uniformly on $X\times\overline\Omega$, and the result follows.
\end{proof}

\subsection{\label{sec:fano} Fano manifolds}

Let $X$ be a Fano manifold. Our goal in this section is to prove that singular K{\"a}hler-Einstein weights, i.e.~weights $\phi\in\cE^1(-K_X)$ such that $\MA(\phi)=e^{-2\phi}$, can be characterized by a variational principle. 

\begin{lem} 
The map $\cE^1(X,\omega)\to L^1(X)$ $\f\mapsto e^{-\f}$ is continuous. 
\end{lem}

\begin{proof} As already observed every $\f\in\cP(X,\omega)$ with full Monge-Amp{\`e}re mass has identically zero Lelong numbers (cf.~\cite{GZ07} Corollary 1.8), which amounts to saying that $e^{-\f}$ belongs to $L^p(X)$ for all $p<+\infty$ by Skoda's integrability criterion. Now let $\f_j\to\f$ be a convergent sequence in $\cE^1(X,\theta)$. Then $e^{-\f_j}\to e^{-\f}$ a.e.. On the other hand $\sup_X\f_j$ is uniformly bounded, thus is follows from the uniform version of Skoda's theorem~\cite{Zeh01} that $e^{-\f_j}$ stays in a bounded subset of $L^2(X)$. In particular $e^{-\f_j}$ is uniformly integrable, and it follows that $e^{-\f_j}\to e^{-\f}$ in $L^1(X)$.  
\end{proof}

Set $L_-(\phi):=-\frac{1}{2}\log\int_Xe^{-2\phi}$ and $F_-:=E-L_-$. Note that $L_-$ is now \emph{concave} on $\cE^1(-K_X)$ by H{\"o}lder's inequality, so that $E-L_-$ is merely the difference of two concave functions. However we have the following psh analogue of Prekopa's theorem, which follows from Berndtsson's results on the psh variation of Bergman kernels
and shows that $L_-$ is \emph{geodesically convex}:

\begin{lem}\label{lem:bernd} 
Let $\Phi:Y\to\cP(-K_X)$ be a psh map. Then $L_-\circ\Phi$ is psh on $Y$. 
\end{lem} 

\begin{proof} Consider the product family $\pi_Y:Z:=X\times Y\to Y$ and consider the line bundle $M:=\pi_X^*(-K_X)$, which coincides with relative anticanonical bundle of $Z/Y$. Then $y\mapsto\frac{1}{2}\log\left(\int_Xe^{-2\Phi(\cdot,y)}\right)^{-1}$ is the weight of the $L^2$ metric induced on the direct image bundle $(\pi_Y)_*\cO_Z\left(K_{Z/Y}+M\right)$. The result thus follows from~\cite{Bern09a}. 
\end{proof}

We are now ready to prove the main part of Theorem D.

\begin{thm}\label{thm:fano} 
Let $X$ be a Fano manifold and let $\phi\in\cE^1(-K_X)$. The following properties are equivalent.
\begin{enumerate} 
\item[(i)] $F_-(\phi)=\sup_{\cE^1(-K_X)}F_-.$
\item[(ii)] $\MA(\phi)=e^{-2\phi+c}$ for some $c\in\R$. 
\end{enumerate}
Furthermore $\phi$ is continuous in that case.
\end{thm}

As mentioned in the introduction this result extends a theorem of Ding-Tian (cf.~\cite{Tian} Corollary 6.26) to singular weights while relaxing the assumption that $H^0(T_X)=0$ in their theorem. 

\begin{proof} The proof of (i)$\Rightarrow$(ii) is similar to that of Theorem~\ref{thm:var}: given $u\in C^0(X)$ we have
$$E(P(\phi+u))+\frac{1}{2}\log\int_X e^{-2(\phi+u)}\le E(P(\phi+u))+\frac{1}{2}\log\int_X e^{-2P(\phi+u)}\le E(\phi)+\frac{1}{2}\log\int_X e^{-2\phi}$$
thus $u\mapsto E(P(\phi+u))+\log\int_X e^{-(\phi+u)}$ achieves its maximum at $0$. By Lemma~\ref{lem:key} $\MA(\phi)$ coincides with the differential of $u\mapsto-\frac{1}{2}\log\int_X e^{-2(\phi+u)}$ at $0$ and we get $\MA(\phi)=e^{-2\phi+c}$ for some $c\in\R$ as desired. 

The equation $\MA(\phi)=e^{-2\phi+c}$ shows in particular that $\MA(\phi)$ has $L^{1+\e}$ density and we infer from
\cite{Kol98} that $\phi$ is continuous. 

Conversely let $\phi\in\cE^1(-K_X)$ be such that $\MA(\phi)=e^{-2\phi+c}$ and let $\psi\in\cE^1(-K_X)$. We are to show that $F_-(\phi)\ge F_-(\psi)$. By scaling invariance of $F_-$ we may assume that $c=0$, and by continuity of $F_-$ along decreasing sequences we may assume that $\psi$ is continuous. Since $\phi$ is also continuous by Kolodziej's theorem, Proposition~\ref{prop:extend} yields a radially symmetric continuous map $\Phi:\overline{A}\to\cP(X,\omega)$ where $A$ denotes the annulus $\{z\in\C,0<\log|z|<1\}$, such that $\Phi$ is harmonic on $A$ and coincides with $\phi$ (resp.~with $\psi$) for $\log|z|=0$ (resp.~$1$). The path $\phi_t:=\Phi(e^t)$ is thus a "continuous geodesic" in $\cP(X,\omega)$, and $E(\phi_t)$ is an affine function of $t$ on the segment $[0,1]$ by Proposition~\ref{prop:flat}. On the other hand Lemma~\ref{lem:bernd} implies that $L_-(\phi_t)$ is a convex function of $t$, thus $F_-(\phi_t)$ is concave, with $F_-(\phi_0)=F_-(\phi)$ and $F_-(\phi_1)=F_-(\psi)$. In order to show that $F_-(\phi)\ge F_-(\psi)$ it will thus be enough to show 
\begin{equation}\label{equ:derF}\frac{d}{dt}_{t=0_+}F_-(\phi_t)\le 0.
\end{equation}
Note that $\phi_t(x)$ is a convex function of $t$ for each $x$ fixed, thus
$$u_t:=\frac{\phi_t-\phi_0}{t}$$
decreases pointwise as $t\to 0_+$ to a function $v$ on $X$ that is bounded from above (by $u_1=\phi_0-\phi_1$). The concavity of $E$ implies
$$\frac{E(\phi_t)-E(\phi_0)}{t}\le\int_Xu_t\MA(\phi_0)$$
hence
\begin{equation}\label{equ:derE}\frac{d}{dt}_{t=0_+}E(\phi_t)\le\int_Xv\MA(\phi_0)=\int_Xv e^{-2\phi_0}\end{equation}
by the monotone convergence theorem (applied to $-u_t$, which is uniformly bounded below and increases to $-v$). Note that this implies in particular that $v\in L^1(X)$. On the other hand we have
$$\frac{\int_X e^{-2\phi_t}-\int_X e^{-2\phi_0}}{t}=-\int_X u_t f(\phi_t-\phi_0)e^{-2\phi_0}$$
with $f(x):=(1-e^{-2x})/x$, and $f(\phi_t-\phi_0)$ is uniformly bounded on $X$ since $\phi_t-\phi_0$ is uniformly bounded. It follows that $|u_t f(\phi_t-\phi_0)|$ is dominated by an integrable function, hence 
\begin{equation}\label{equ:derL}\frac{d}{dt}_{t=0_+}\int_Xe^{-2\phi_t}=-\int_X v e^{-2\phi_0}
\end{equation}
since $f(\phi_t-\phi_0)\to 1$. The combination of (\ref{equ:derE}) and (\ref{equ:derL}) now yields (\ref{equ:derF}) as desired.
\end{proof}

\begin{rem} 
Suppose that $\phi,\psi\in\cP(-K_X)$ are smooth such that $\phi$ is K{\"a}hler-Einstein. We would like to briefly sketch Ding-Tian's 
argument for comparison. Since $F_-$ is translation invariant we may assume that they are normalized so that $\int_X e^{-2\phi}=\int_Xe^{-2\psi}=0$, and our goal is to show that $E(\phi)\ge E(\psi)$. 
By the normalization we get $\MA(\phi)=Ve^{-2\phi}$ with $V:=\vol(-K_X)=c_1(X)^n$ 
and there exists a smooth weight $\tau\in\cP(-K_X)$ such that $\MA(\tau)=Ve^{-2\psi}$ by~\cite{Yau78}. If we further assume that $H^0(T_X)=0$ then~\cite{BM87} yields the existence of a \emph{smooth} path $\phi_t\in\cP(-K_X)\cap C^\infty$ with $\phi_0=\tau$, $\phi_1=\phi$ and 
\begin{equation}\label{equ:MAt}\MA(\phi_t)=Ve^{-(t\phi_t+(1-t)\psi)}
\end{equation}
for each $t\in[0,1]$. The argument of Ding-Tian can then be formulated as follows. The claim is that $t(E(\phi_t)-E(\psi))$ is a non-decreasing function of $t$, which implies $E(\phi)-E(\psi)\ge 0$ as desired. Indeed we have
\begin{equation}\label{equ:derive}\frac{d}{dt}\left(t(E(\phi_t)-E(\psi))\right)=E(\phi_t)-E(\psi)+t\langle E'(\phi_t),\dot{\phi}_t\rangle.
\end{equation}
On the other hand differentiating $\int_X e^{-(t\phi_t+(1-t)\psi)}=1$ yields 
$$0=\frac{d}{dt}\int_X e^{-(t\phi_t+(1-t)\psi)}=-\int_X(\phi_t+t\dot{\phi}_t-\psi)e^{-(t\phi_t+(1-t)\psi)}$$
thus 
$$\langle E'(\phi_t),\phi_t+t\dot{\phi}_t-\psi\rangle=0$$
by (\ref{equ:MAt}), and (\ref{equ:derive}) becomes 
$$\frac{d}{dt}\left(t(E(\phi_t)-E(\psi))\right)=E(\phi_t)-E(\psi)+\langle E'(\phi_t),\psi-\phi_t\rangle=J_{\phi_t}(\psi)$$
which is non-negative as desired by concavity of $E$. 
\end{rem}

\noindent {\bf Proof of Theorem D.}
The first part of the proof of Theorem D follows from Theorem \ref{thm:fano}.
Observe now that $F_-$ is u.s.c. If it is $J$-proper then its supremum is attained on a compact
convex set of weights with energy uniformly bounded from below by some large constant $-C$.  
The conclusion thus follows from Theorem \ref{thm:fano}. 

\smallskip

As opposed to $F_+$, let us recall for emphasis that $F_-$ is not necessarily $J$-proper (see~\cite{Tian}).

\subsection{\label{sec:unique} Uniqueness of K{\"a}hler-Einstein metrics}

This section is devoted to the proof of Theorem E, which extends in particular~\cite{BM87} in case $H^0(T_X)=0$. 

\begin{thm}\label{thm:unique} Let $X$ be a K{\"a}hler-Einstein Fano manifold without non-trivial holomorphic vector field. Then $F$ achieves its maximum on $\cT^1(-K_X)$ at a unique point. 
\end{thm}
\begin{proof} Let $\phi$ be a smooth K{\"a}hler-Einstein weight on $-K_X$,
  which exists by assumption. We may assume that $\phi$ is normalized
  so that $\MA(\phi)=e^{-2\phi}$. Now let $\psi\in\cE^1(-K_X)$ be such that $\MA(\psi)=e^{-\psi}$. We are going to show that $\phi=\psi$. By Kolodziej's theorem $\psi$ is continuous, and we consider as before the continuous geodesic $\phi_t$ connecting $\phi_0=\phi$ to $\phi_1=\psi$. Theorem~\ref{thm:fano} implies that the concave function $F_-(\phi_t)$ achieves its maximum at $t=0$ and $t=1$, thus $F_-(\phi_t)$ is \emph{constant} on $[0,1]$. Since $E(\phi_t)$ is affine, it follows that $L_-(\phi_t)$ is also affine on $[0,1]$, hence $L_-(\phi_t)\equiv 0$ since $L_-(\phi_0)=L_-(\phi_1)=0$ by assumption. This implies in turn that $E(\phi_t)$ is constant. Theorem~\ref{thm:fano} therefore yields $\MA(\phi_t)=e^{-2\phi_t}$ for all $t\in[0,1]$. 

Set $v_t:=\frac{\partial}{\partial t}\phi_t$, which is
non-decreasing in $t$ by convexity. One sees as in the proof of Theorem~\ref{thm:fano} that $v_t\in L^1(X)$ and 
\begin{equation}\label{equ:int}\int_X v_t e^{-2\phi_t}=0
\end{equation}
for all $t$. We claim that $v_0=0$, which will imply $v_t\ge v_0=0$ for all $t$, hence $v_t=0$ a.e. for all $t$ by (\ref{equ:int}), and the proof will be complete. 

We are going to show by differentiating the equation
$(dd^c\phi_t)^n=e^{-2\phi_t}$ that
\begin{equation}\label{equ:laplace}
ndd^cv_0\wedge(dd^c\phi_0)^{n-1}=-v_0 e^{-2\phi_0}\end{equation}
in the sense of distributions, i.e.
$$n\int_X v_0(dd^c\phi_0)^{n-1}\wedge dd^c w=-\int_X w
v_0(dd^c\phi_0)^n$$
for every smooth function $w$ on $X$. Using
$(dd^c\phi_0)^n=e^{-2\phi_0}$ (\ref{equ:laplace}) means that $v_0$ is
an eigendistribution with eigenvalue $-1$ of the Laplacian $\Delta$ of
the (smooth) K{\"a}hler-Einstein metric $dd^c\phi_0$, and thus $v_0=0$
since $H^0(T_X)=0$ (cf.\cite{Tian}, Lemma
6.12). 

We claim that 
\begin{equation}\label{equ:derexp}\frac{d}{dt}_{t=0_+}\int_X w e^{-2\phi_t}=-\int_X w v_0
e^{-2\phi_0}
\end{equation}
and \begin{equation}\label{equ:derma}\frac{d}{dt}_{t=0_+}\int_X
w(dd^c\phi_t)^n=n\int_Xv_0(dd^c\phi_0)^{n-1}\wedge dd^c w,
\end{equation}
 which will imply (\ref{equ:laplace}). The proof of (\ref{equ:derexp}) is handled as before: we write
$$\int_X w\frac{e^{-2\phi_t}-e^{-2\phi_0}}{t}=-\int_X w u_t
f(\phi_t-\phi_0)e^{-2\phi_0}$$
with $f(x):=(1-e^{-x})/x$ and use the monotone convergence theorem. 

On the other hand, writing $dd^cw$ as the difference of two positive $(1,1)$-forms shows by monotone convergence that (\ref{equ:derma}) is equivalent to
$$\int_Xw\left((dd^c\phi_t)^n-(dd^c\phi_0)^n\right)=n\int_X(\phi_t-\phi_0)(dd^c\phi_0)^n\wedge dd^cw+o(t),$$
where the left-hand side can be rewritten as
$$\int_X(\phi_t-\phi_0)\left(\sum_{j=0}^{n-1}(dd^c\phi_t)^j\wedge(dd^c\phi_0)^{n-j-1}\right)\wedge
dd^cw$$
after integration by parts. The result will thus follow if we can show that
$$\int_X(\phi_t-\phi_0)\left((dd^c\phi_t)^j\wedge(dd^c\phi_0)^{n-j-1}-(dd^c\phi_0)^{n-1}\right)\wedge
dd^cw=o(t)$$
for $j=0,...,n-1$, which will in turn follow from
\begin{equation}\label{equ:o}\int_X(\phi_t-\phi_0)dd^c(\phi_t-\phi_0)\wedge(dd^c\phi_t)^j\wedge(dd^c\phi_0)^{n-j-2}\wedge
dd^cw=o(t)\end{equation}
for $j=0,...,n-2$. Now we have
$$\int_X(\phi_t-\phi_0)dd^c(\phi_t-\phi_0)\wedge(dd^c\phi_t)^j\wedge(dd^c\phi_0)^{n-j-2}\wedge
dd^cw$$
$$=\int_X d(\phi_t-\phi_0)\wedge d^c(\phi_t-\phi_0)\wedge(dd^c\phi_t)^j\wedge(dd^c\phi_0)^{n-j-2}\wedge
dd^cw.$$
Since $w$ is smooth and $dd^c\phi_0$ is a K{\"a}hler form we have 
$$-Cdd^c\phi_0\le dd^cw\le C dd^c\phi_0$$
for $C\gg 1$, and we see that (\ref{equ:o}) will follow from 
$$\int_X d(\phi_t-\phi_0)\wedge
d^c(\phi_t-\phi_0)\wedge(dd^c\phi_t)^j\wedge(dd^c\phi_0)^{n-j-1}=o(t)$$
for $j=0,...,n-1$ since $d(\phi_t-\phi_0)\wedge
d^c(\phi_t-\phi_0)\wedge(dd^c\phi_t)^j\wedge(dd^c\phi_0)^{n-j-2}$ is a
positive current. Now consider as before
$$\Delta_{\phi_0}E(\phi_t):=E(\phi_0)-E(\phi_t)+\int_X(\phi_t-\phi_0)\MA(\phi_0).$$
Since $E(\phi_t)$ is constant, the monotone convergence theorem yields
$$\frac{d}{dt}_{t=0_+}\Delta_{\phi_0}E(\phi_t)=\int_X
v_0\MA(\phi_0)=\int_X v_0 e^{-2\phi_0}=0.$$
By (\ref{equ:delta}) this implies that
$$\int_Xd(\phi_t-\phi_0)\wedge
d^c(\phi_t-\phi_0)\wedge(dd^c\phi_t)^j\wedge(dd^c\phi_0)^{n-j-1}=o(t)$$
for $j=0,...,n-1$ as desired.
\end{proof}

\section{Balanced metrics}\label{sec:balanced}

Let $A$ be an ample line bundle and denote by $\cH_k$ the space of all positive Hermitian products on $H^0(kA)$, which is isomorphic to the Riemannian symmetric space
$$\cH_k\simeq GL(N_k,\C)/U(N_k)
$$ 
with $N_k:=h^0(kA)$. We will always assume that $k$ is taken large enough to ensure that $kA$ is very ample. There is a natural injection
$$\FS_k:\cH_k\hookrightarrow\cP(A)\cap C^\infty$$ 
sending $H\in\cH_k$ to the Fubiny-Study type weight 
$$\FS_k(H):=\frac{1}{2k}\log\left(\frac{1}{N_k}\sum_{j=1}^{N_k}|s_j|^2\right)$$
where $(s_j)$ is an $H$-orthonormal basis of $H^0(kA)$. 

On the other hand every measure $\mu$ on $X$ yields a map
$$\Hilb_k(\mu,\cdot):\cP(A)\to\cH_k$$
by letting $\Hilb_k(\mu,\phi)$ be the $L^2$-scalar product on $H^0(kA)$ induced by $\mu$ and $k\phi$. 

We are going to consider the following three situations (compare~\cite{Don05b}).\\

\noindent({\bf S}$_\mu$) Let $\mu$ be a probability measure with finite energy on $X$ and let $\phi_0$ be a reference smooth strictly psh weight on $A$. We set 
$$
\Hilb_k(\phi):=\Hilb_k(\mu,\phi)
$$
and 
$$L(\phi):=L_\mu(\phi)=\int_X(\phi-\phi_0)d\mu.
$$
We also let $T\in c_1(A)$ be the unique closed positive current with finite energy such that $V^{-1}\langle T^n\rangle=\mu$ where $V:=(A^n)$.\\

\noindent({\bf S}$_+$) $A=K_X$ is ample. A weight $\phi\in\cP(K_X)$ induces a measure $e^{-2\phi}$ with $L^\infty$ density on $X$ and we set
$$
\Hilb_k(\phi):=\Hilb_k(e^{2\phi},\phi)
$$
and 
$$
L(\phi):=L_+(\phi)=\frac{1}{2}\log\int_X e^{2\phi}.
$$
We let $T:=\omega_{KE}$ be the unique K\"ahler-Einstein metric.\\

\noindent({\bf S}$_-$) $A=-K_X$ is ample. A weight $\phi\in\cE^1(-K_X)$ induces a measure $e^{-2\phi}$ on $X$ with $L^p$ density for all $p<+\infty$ and we set
$$
\Hilb_k(\phi):=\Hilb_k(e^{-2\phi},\phi)
$$
and 
$$
L(\phi):=L_-(\phi)=-\frac{1}{2}\log\int_X e^{-2\phi}.
$$
We also assume that $H^0(T_X)=0$ and that $T:=\omega_{KE}$ is a K\"ahler-Einstein metric, which is therefore unique by~\cite{BM87} or Theorem~\ref{thm:unique} above.\\

As in~\cite{Don05b} we shall say in each case that $H\in\cH_k$ is \emph{$k$-balanced} if it is a fixed point of $\Hilb_k\circ\FS_k$. The maps $\Hilb_k$ and $\FS_k$ induce a bijective correspondence between the $k$-balanced point in $\cH_k$ and the \emph{$k$-balanced weights} $\phi\in\cP(A)$, i.e. the fixed point of $\FS_k\circ\Hilb_k$. The $k$-balanced points $H\in\cH_k$ admit the following variational characterization (cf.~\cite{Don05b} and Corollary~\ref{cor:balvar} below). Consider the function $D_k$ on $\cH_k$ defined by
\begin{equation}\label{equ:det}
D_k:=-\frac{1}{2kN_k}\log\det,
\end{equation}
where the determinant is computed with respect to a fixed base point in $\cH_k$. Then $H\in\cH_k$ is $k$-balanced iff it maximizes the function 
\begin{equation}\label{equ:fk}
F_k:=D_k-L\circ\FS_k
\end{equation}
on $\cH_k$. There exists furthermore at most one such maximizer up to scaling (Corollary~\ref{cor:concave}).

Our main result in this section is the following. 
\begin{thm}\label{thm:bal} In each of the three settings $({\bf S}_\mu)$, $({\bf S}_+)$ and $({\bf S}_-)$ above there exists for each $k\gg 1$ a $k$-balanced metric $\phi_k\in\cP(A)$, unique up to a constant. Moreover in each case $dd^c\phi_k$ converges weakly to $T$ as $k\to\infty$. 
\end{thm}
This type of result has its roots in the seminal work of Donaldson~\cite{Don01} and the present statements were inspired by~\cite{Don05b}. In fact the existence of $k$-balanced metrics in case ({\bf S}$_\mu$) was established in~\cite{Don05b} Proposition 3 assuming that $\mu$ integrates $\log|s|$ for every section $s\in H^0(mA)$. On P.12 of the same paper the author conjectured the convergence statement in the case where $\mu$ is a smooth positive volume form, by analogy with~\cite{Don01}. The result was indeed observed to hold for such measures in~\cite{Kel09} as a direct consequence of the work of Wang \cite{Wan05}, which in turn relied on the techniques introduced in~\cite{Don01}. The settings ({\bf S}$_\pm$) were introduced and briefly discussed in Section 2.2.2 of~\cite{Don05b}.\\

The main idea of our argument goes as follows. In each case the functional $F:=E-L$ is usc and $J$-coercive on $\cE^1(A)$ (by Corollary~\ref{cor:proper} in case ({\bf S}$_\mu$) and ({\bf S}$_+$) and by~\cite{PSSW08} in case ({\bf S}$_-$)) and $T$ is characterized as the unique maximizer of $F$ on $\cT^1(A)=\cE^1(A)/\R$ by our variational results. 

The crux of our proof is Lemma~\ref{lem:estim} below which compares the restriction $J\circ f_k$ of the exhaustion function of $\cE^1(A)$ to $\cH_k$ to a natural exhaustion function $J_k$ on $\cH_k$.
This result enables us to carry over the $J$-coercivity of $F$ to a $J_k$-coercivity property of $F_k$ that is furthermore uniform with respect to $k$ (Lemma~\ref{lem:proper}). This shows on the one hand that $F_k$ achieves its maximum on $\cH_k$, which yields the existence of a $k$-balanced weight $\phi_k$. On the other hand it provides a lower bound 
$$
F(\phi_k)\ge\sup_{\cH_k}F_k+o(1)
$$
which allows us to show that $\phi_k$ is a maximizing sequence for $F$. We can then use 
Proposition~\ref{prop:proper} to conclude that $dd^c\phi_k$ converges to $T$.

\subsection{Convexity properties}
Any geodesic $t\mapsto H_t$ in $\cH_k$ is the image of $1$-parameter subgroup of $GL(H^0(kA))$, which means that there exists a basis $S=(s_j)$ of $H^0(kA)$ and
$$(\lambda_1,...,\lambda_{N_k})\in\R^{N_k}
$$
such that $e^{\lambda_j t}s_j$ is  $H_t$-orthonormal for each $t$. We will say that $H_t$ is \emph{isotropic} if 
$$
\lambda_1=...=\lambda_{N_k}.
$$
The isotropic geodesics are thus the orbits of the action of $\R_+$ on $\cH_k$ by scaling. In these notations there exists $c\in\R$ such that
\begin{equation}\label{equ:detk}
D_k(H_t)=\frac{t}{kN_k}\sum_j\lambda_j+c
\end{equation}
for all $t$, and we have 
\begin{equation}\label{equ:fst}\FS_k(H_t)=\frac{1}{2k}\log\left(\frac{1}{N_k}\sum_je^{2t\lambda_j}|s_j|^2\right). 
\end{equation}
Observe that $z\mapsto\FS_k(H_{\Re z})$ defines a \emph{psh map} $\C\to\cP(A)$, i.e. 
$\FS_k(H_{\Re z})$ is psh in all variables over $\C\times X$. We also record the formula

\begin{equation}\label{equ:fubder}\frac{\partial}{\partial t}\FS_k(H_t)=\frac{1}{k}\frac{\sum_j \lambda_j e^{2t\lambda_j}|s_j|^2}{\sum_j e^{2t\lambda_j}|s_j|^2}.
\end{equation}

The next convexity properties will be crucial to the proof of Theorem~\ref{thm:bal}. Recall that $k$ is assumed to be large enough to guarantee that $kA$ is very ample.

\begin{lem}\label{lem:conv} The function $D_k$ is affine on $\cH_k$ and $E\circ\FS_k$ is convex. Moreover in each of the three settings $({\bf S}_\mu)$, $({\bf S}_+)$ and $({\bf S}_-)$ above $L\circ\FS_k$ is convex on $\cH_k$, and \emph{strictly} convex along non-isotropic geodesics. 
\end{lem} 
\begin{proof} The first property follows from (\ref{equ:detk}). Let $H_t$ be a geodesic in $\cH_k$ and set 
$$
\phi_t:=\FS_k(H_t).
$$
The convexity of $t\mapsto E(\phi_t)$ follows from Proposition~\ref{prop:flat} since $z\mapsto\phi_{\Re z}$ is a psh map as was observed above. 

Let us now first consider the cases ({\bf S}$_\mu$) and ({\bf S}$_+$). Since $t\mapsto\phi_t(x)$ is convex for each $x\in X$ the convexity of $L(\phi_t)$ directly follows since $\phi\mapsto L(\phi)$ is convex and non-decreasing in these cases. In order to get the strict convexity along non-isotropic geodesics one however has to be slightly more precise. By (\ref{equ:fubder}) we have
$$
k\frac{\partial}{\partial t}\phi_t=\sum_j\lambda_j\sigma_j(t)
$$
with
$$
\sigma_j(t):=\frac{e^{2t\lambda_j}|s_j|^2}{\sum_i e^{2t\lambda_i}|s_i|^2},
$$
and a computation yields
$$
\frac{k}{2}\frac{\partial^2}{\partial t^2}\phi_t=\left(\sum_j\lambda_j^2\sigma_j(t)\right)-\left(\sum_j\lambda_j\sigma_j(t)\right)^2.
$$
Now the Cauchy-Schwarz inequality implies that 
$$
\left(\sum_j\lambda_j\sigma_j(t)\right)^2\le\left(\sum_j\lambda_j^2\sigma_j(t)\right)\left(\sum_j\sigma_j(t)\right),
$$
which shows that $\frac{\partial^2}{\partial t^2}\phi_t\ge 0$ (which we already knew) since
$$
\sum_j\sigma_j(t)=1.
$$
Furthermore the equality case $\frac{\partial^2}{\partial t^2}\phi_t(x)=0$ holds for a given $t\in\R$ and a given $x\in X$ iff there exists $c\in\R$ such that for all $j$ we have
$$
\lambda_j\sigma_j(t)^{1/2}=c\sigma_j(t)^{1/2}
$$
at the point $x$. If $x$ belongs to the complement of the zero divisors $Z_1,...,Z_{N_k}$ of the $s_j$'s we therefore conclude that $\frac{\partial^2}{\partial t^2}\phi_t(x)>0$ for all $t$ unless $H_t$ is isotropic. 

Now in both cases ({\bf S}$_\mu$) and ({\bf S}$_+$) the map $\phi\mapsto L(\phi)$ is convex and non-decreasing on $\cP(A)$ as we already noticed. We thus have 
$$
\frac{d^2}{dt^2}L(\phi_t)\ge\int_X\left(\frac{\partial^2}{\partial t^2}\phi_t \right)L'(\phi_t)
$$
where $L'(\phi_t)$ is viewed as a positive measure on $X$. This measure is in both cases non-pluripolar,  thus the union of the zero divisors $Z_j$ has zero measure with respect to $L'(\phi_t)$, and it follows as desired from the above considerations that $t\mapsto L(\phi_t)$ is strictly convex when $H_t$ is non-isotropic. 

We finally consider case ({\bf S}$_-$). Since $z\mapsto\phi_{\Re z}$ is a psh map, the convexity of $t\mapsto L(\phi_t)$ follows from Lemma~\ref{lem:bernd}, which was itself a direct consequence of~\cite{Bern09a}. Now if we assume that $H_t$ is non-isotropic then the strict convexity follows from~\cite{Bern09b}. Indeed if $t\mapsto L_-(\phi_t)$ is affine on a non-empty open interval $I$ then~\cite{Bern09b} Theorem 2.4 implies that $c(\phi_t)=0$ on $I$ and that the vector field $V_t$ that is dual to the $(0,1)$-form
$$
\overline{\partial}\left(\frac{\partial}{\partial t}\phi_t\right)
$$
with respect to the metric $dd^c\phi_t$ is holomorphic for each $t\in I$. Since we assume that $H^0(T_X)=0$ we thus get $V_t=0$. But we have by definition
$$
c(\phi_t)=\frac{\partial^2}{\partial t^2}\phi_t-|V_t|^2
$$
where the norm of $V_t$ is computed with respect to $dd^c\phi_t$, and we conclude that $\frac{\partial^2}{\partial t^2}\phi_t=0$ on $I$. This however implies that $H_t$ is isotropic by the first part of the proof, and we have reached a contradiction. 
\end{proof}

\begin{cor}\label{cor:concave} The function $F_k:=D_k-L\circ\FS_k$ is concave on $\cH_k$ and all its critical points are proportional.
\end{cor}
\begin{proof} The first assertion follows directly from Lemma~\ref{lem:conv}. As a consequence $H\in\cH_k$ is a criticical point of $F_k$ iff it is a maximizer. Now let $H_0,H_1$ be two critical points and let $H_t$ be the geodesic through $H_0,H_1$. If $H_t$ is non-isotropic then $t\mapsto F_k(H_t)$ is \emph{strictly} concave, which contradicts the fact that it is maximized at $t=0$ and $t=1$. So we conclude that $H_t$ must be isotropic, which means that $H_0$ and $H_1$ are proportional as desired.
\end{proof}

\subsection{Variational characterization of balanced metrics}
Recall that a $k$-balanced weight $\phi$ is by definition a fixed point of 
$\FS_k\circ\Hilb_k$. The maps $\FS_k$ and $\Hilb_k$ induce a bijective correspondence between the fixed points of $\FS_k\circ\Hilb_k$ and those of $t_k:=\Hilb_k\circ\FS_k$ in $\cH_k$. 

The following result is implicit in \cite{Don05b}. 

\begin{lem}\label{lem:caracbal} Let $H\in\cH_k$. Then $H$ is a fixed point of $t_k$ iff it is a critical point of 
$$
F_k=D_k-L\circ\FS_k.
$$
\end{lem}
\begin{proof} Recall that for each geodesic $H_t$ with $H_0=H$ there exists $\lambda\in\R^{N_k}$ and an $H$-orthonormal basis $(s_j)$ such that $e^{t\lambda_j}s_j$ is $H_t$-orthonormal. We claim that 
\begin{equation}\label{equ:deriv}
k\frac{d}{dt}_{t=0}L\circ\FS_k(H_t)=\left(\sum_j\lambda_j\|s_j\|^2_{t_k(H)}\right)\left(\sum_j\|s_j\|^2_{t_k(H)}\right)^{-1}.
\end{equation}
In case ({\bf S}$_\mu$) we have by (\ref{equ:fubder}) 
$$
k\frac{d}{dt}_{t=0}L\circ\FS_k(H_t)=\int_X\frac{\sum_j\lambda_j|s_j|^2}{\sum_j|s_j|^2}d\mu
$$
$$
=\sum_j\lambda_j\int_X |s_j|^2e^{-2k\FS_k(H)}d\mu=\sum_j\lambda_j\|s_j\|^2_{\Hilb_k\circ\FS_k(H)}
$$
and the result follows since
$$
\sum_j\|s_j\|^2_{\Hilb_k\circ\FS_k(H)}=1
$$
in that case. In case ({\bf S}$_\pm$) we find on the other hand
$$
k\frac{d}{dt}_{t=0}L\circ\FS_k(H_t)=\left(\int_X\frac{\sum_j\lambda_j|s_j|^2}{\sum_j|s_j|^2}e^{\pm 2\FS_k(H)}\right)\left(\int_X e^{\pm 2\FS_k(H)}\right)^{-1}.
$$
and (\ref{equ:deriv}) again easily follows by writing
$$
\int_X e^{\pm 2\FS_k(H)}=\sum_j\int_X\frac{|s_j|^2}{\sum_i|s_i|^2}e^{\pm 2\FS_k(H)}=\sum_j\|s_j\|^2_{t_k(H)}.
$$

As a consequence of (\ref{equ:deriv}) we see that $H$ is a critical point of $F_k=D_k-L\circ\FS_k$ iff
\begin{equation}\label{equ:crit}\frac{1}{N_k}\sum_j\lambda_j=\left(\sum_j\lambda_j\|s_j\|^2_{t_k(H)}\right)\left(\sum_j\|s_j\|^2_{t_k(H)}\right)^{-1}
\end{equation}
holds for all $H$-orthonormal basis $(s_j)$ and all $\lambda\in\R^{N_k}$. If we choose in particular $(s_j)$ to be also $t_k(H)$-orthogonal then (\ref{equ:crit}) holds for all $\lambda\in\R^{N_k}$ iff $\|s_j\|^2_{t_k(H)}=1$ for all $j$, which means that $t_k(H)=H$. Conversely $t_k(H)=H$ certainly implies (\ref{equ:crit}) since $(s_j)$ is then $t_k(H)$-orthonormal, and the proof is complete.
\end{proof} 
As a consequence of Corollary~\ref{cor:concave} and Lemma~\ref{lem:caracbal} we get

\begin{cor}\label{cor:balvar} Up to an additive constant there exists at most one $k$-balanced weight $\phi\in\cP(A)$, and $\phi$ exists iff $F_k=D_k-L\circ\FS_k$ admits a maximizer $H\in\cH_k$, in which case we have $\phi=\FS_k(H)$.
\end{cor}

\subsection{Asymptotic comparison of exhaustion functions}
Recall that we have fixed a reference smooth strictly psh weight $\phi_0$ on $A$. We set $\mu_0:=\MA(\phi_0)$ and normalize the determinant (and thus the function $D_k$) by taking
$$
B_k:=\Hilb_k(\mu_0,\phi_0)
$$
as a base point in $\cH_k$ and setting $\det B_k=1$. 

We now introduce a natural exhaustion function on $\cH_k/\R_+$. 
\begin{lem}\label{lem:exhaust} The scale-invariant function $J_k:=L_0\circ\FS_k-D_k$ induces a convex exhaustion function of $\cH_k/\R_+$.
\end{lem}
\begin{proof} Convexity follows from Lemma~\ref{lem:conv}. The fact that $J_k\to+\infty$ at infinity on $\cH_k/\R_+$ is easily seen and is a special case of~\cite{Don05b} Proposition 3. 
\end{proof}

The next key estimate shows that the restriction $J\circ\FS_k$ of the exhaustion function $J$ of $\cE^1(A)$ to $\cH_k$ is asymptotically bounded from above by the exhaustion function $J_k$. In other words the injection
$$
\FS_k:\cH_k\hookrightarrow\cE^1(A)
$$
sends each $J_k$-sublevel set $\{J_k\le C\}$ into a $J$-sublevel set $\{J\le C_k\}$ where $C_k$ is only slighly larger than $C$. 

\begin{lem}\label{lem:estim} There exists $\e_k\to 0$ such that
\begin{equation}\label{equ:estim}
J\circ\FS_k\le(1+\e_k)J_k+\e_k\text{ on }\cH_k
\end{equation}
for all $k$. 
\end{lem}

Before proving this result we need some preliminaries. Given any weight $\phi$ on $A$ recall that the \emph{distortion function} of $(\mu_0,k\phi)$ is defined by 
$$
\rho_k(\mu_0,\phi):=\sum_j|s_j|^2_{k\phi}
$$
where $(s_j)$ is an arbitrary $\Hilb_k(\mu_0,\phi)$-orthonormal basis of $H^0(kA)$, and the \emph{Bergman measure} of $(\mu_0,k\phi)$ is then the proability measure
$$
\beta_k(\mu_0,\phi):=\frac{1}{N_k}\rho_k(\mu_0,\phi)\mu_0.
$$
When $\phi$ is smooth and strictly psh the Bouche-Catlin-Tian-Zelditch theorem~\cite{Bou90,Cat99, Tia90, Zel98} implies the $C^\infty$-convergence
\begin{equation}\label{equ:berg}
\lim_{k\to\infty}\beta_k(\mu_0,\phi)=\MA(\phi).
\end{equation}
The operator 
$$
P_k:=\FS_k\circ\Hilb_k(\mu_0,\cdot)
$$
satisfies by definition
$$
P_k(\phi)-\phi=\frac{1}{2k}\log\left(N_k^{-1}\rho_k(\mu_0,\phi)\right).
$$
As a consequence any smooth strictly psh weight $\phi$ is the $C^\infty$ limit of $P_k(\phi)$.

Now let $H\in\cH_k$ and let $t\mapsto H_t$ be the (unique) geodesic in $\cH_k$ such that $H_0=B_k$ and $H_1=H$. We denote by
$$
v(H):=\frac{\partial}{\partial t}_{t=0}\FS_k(H_t)
$$
the tangent vector at $t=0$ to the corresponding path $t\mapsto\FS_k(H_t)$. As before there exists $(\lambda_1,...,\lambda_{N_k})\in\R^{N_k}$ and a basis $(s_j)$ that is both $B_k$-orthonormal and $H$-orthogonal  such that
\begin{equation}\label{equ:tgt}v(H)=\frac{1}{k}\frac{\sum_j\lambda_j |s_j|^2}{\sum_j|s_j|^2}.
\end{equation}
By convexity in the $t$-variable we note that
\begin{equation}\label{equ:tangent}v(H)\le\FS_k(H_1)-\FS_k(H_0)=\FS_k(H)-P_k(\phi_0)
\end{equation}
holds pointwise on $X$.  
  
\begin{lem}\label{lem:deriv} We have
$$D_k(H)=\int_Xv(H)\beta_k(\mu_0,\phi_0)$$
\end{lem}
\begin{proof} Let $H_t$ be the geodesic through $B_k$ and $H$ as above. On the one hand we have
$$
D_k(H_t)=\frac{t}{kN_k}\sum_j\lambda_j.
$$ 
On the other hand (\ref{equ:tgt}) yields
$$
\int_Xv(H)\beta_k(\mu_0,\phi_0)=\frac{1}{kN_k}\sum_j\lambda_j\int_X|s_j|^2_{k\phi_0}d\mu_0
$$
and the result follows since $(s_j)$ is $B_k$-orthonormal.
\end{proof}

We are now in a position to prove Lemma~\ref{lem:estim}.\\

\noindent{\bf Proof of Lemma~\ref{lem:estim}}. Let $H\in\cH_k$. In what follows all $O$ and $o$ are meant to hold as $k\to\infty$ uniformly with respect to $H\in\cH_k$. By scaling invariance of both sides of (\ref{equ:estim}) we may assume that $H$ is normalized by 
$$
L_0(\FS_k(H))=0,
$$
so that
$$
\sup_X(\FS_k(H)-\phi_0)\le O(1)
$$
and (\ref{equ:tangent}) yields
\begin{equation}\label{equ:maj}\sup_X v(H)\le O(1).
\end{equation}
since $P_k(\phi_0)=\phi_0+O(1)$.
 
On the other hand Lemma~\ref{lem:deriv} gives
\begin{equation}\label{equ:elk}
D_k(H)=\int_X v(H)\mu_0+o\left(\|v(H)\|_{L^1}\right)
\end{equation}
since $\b_k(\mu_0,\psi_0)\to\MA(\psi_0)=\mu_0$ in $L^\infty$ by Bouche-Catlin-Tian-Zelditch. Now we have
$$
\|v(H)\|_{L^1}\le 2\sup_X v(H)-\int_X v(H)d\mu_0
$$
$$
=-D_k(H)+o(\|v(H)\|_{L^1})+O(1)
$$
(by (by (\ref{equ:maj})) and (\ref{equ:elk})) and it follows that
\begin{equation}\label{equ:ellin}
(1+o(1))\|v(H)\|_{L^1}\le-D_k(H)+O(1).
\end{equation}
On the other hand the convexity of $E\circ\FS_k$ (Lemma~\ref{lem:conv}) shows that
$$
E\circ\FS_k(H)-E(P_k(\phi_0))\ge\langle E'(P_k(\phi_0)),v(H)\rangle=\int_Xv(H)\MA(P_k(\psi_0)).
$$
Now we have $E(P_k(\phi_0))=o(1)$ since $P_k(\phi_0)=\phi_0+o(1)$ uniformly on $X$ and 
$$
\int_Xv(H)\MA(P_k(\psi_0))=\int_X v(H)\mu_0+o(\|v(H)\|_{L^1})
$$
by uniform convergence of $\MA(P_k(\psi_0))$ to $\MA(\psi_0)=\mu_0$. By (\ref{equ:elk}) we thus get
$$
E\circ\FS_k(H)\ge D_k(H)+o(\|v(H)\|_{L^1})+o(1)
$$

$$
\ge (1+o(1))D_k(H)+o(1)
$$
by (\ref{equ:ellin}) and the result follows. 

\subsection{Coercivity} 
Recall that $F=E-L$ is $J$-coercive, i.e.~there exists $0<\de<1$ and $C>0$ such that 
\begin{equation}\label{equ:coerc}
F\le-\de J+C
\end{equation}
on $\cE^1(A)$. 
The next result uses the  key estimate (\ref{equ:estim}) to show that the $J$-coercivity of $F$ carries over to a uniform $J_k$-coercivity estimate for $F_k=D_k-L\circ\FS_k$ for all $k\gg 1$. 

\begin{lem}\label{lem:proper} There exists $\e>0$ and $B>0$ such that
$$
F_k\le-\e J_k+B
$$
holds on $\cH_k$ for all $k\gg 1$. 
\end{lem} 
\begin{proof} As discussed after Definition~\ref{defi:proper} (\ref{equ:coerc}) is equivalent to the linear upper bound
\begin{equation}\label{equ:linear}L_0-L\le(1-\de)J+C
\end{equation}
which implies
$$
L_0\circ\FS_k-L\circ\FS_k\le(1-\de)J\circ\FS_k+C.
$$
On the other hand we have
$$
J\circ\FS_k\le(1+\e_k)J_k+\e_k
$$
by (\ref{equ:estim}) hence
$$
L_0\circ\FS_k-L\circ\FS_k\le(1-\de)(1+\e_k)J_k+C+\e_k.
$$
Since $J\ge 0$ (\ref{equ:estim}) shows in particular that $J_k$ bounded below on $\cH_k$ uniformly with respect to $k$. For $k\gg 1$ we have $(1-\de)(1+\e_k)<(1-\e)$ and $C+\e_k<B$ for some $\e>0$ and $B>0$ and we thus infer
$$
L_0\circ\FS_k-L\circ\FS_k\le(1-\e)J_k+B. 
$$
It is then immediate to see that this is equivalent to the desired inequality by using $J_k=L_0\circ\FS_k-D_k$. 
\end{proof}
Note that the coercivity constants $\e$ and $B$ of $F_k$ can even be taken arbitrarily close to those $\de$ and $C$ of $F$, as the proof shows.

Combining Lemma~\ref{lem:proper} with Lemma~\ref{lem:exhaust} yields
\begin{cor}\label{cor:properk} For each $k\gg 1$ the scale-invariant functional $F_k$ tends to $-\infty$ at infinity on $\cH_k/\R_+$, hence it achieves its maximum on $\cH_k$.  
\end{cor}

\subsection{Proof of Theorem~\ref{thm:bal}}
The existence and uniqueness of a $k$-balanced metric $\phi_k$ for $k\gg 1$ follows by combining Corollary~\ref{cor:balvar} and Corollary~\ref{cor:properk}. Recall that $\phi_k=\FS_k(H_k)$ where $H_k\in\cH_k$ is the unique maximizer of $F_k=D_k-L\circ\FS_k$ on $\cH_k$. 

In order to prove the convergence of $dd^c\phi_k$ to $T$ we will rely on Proposition~\ref{prop:proper}. Since $T$ is characterized as the unique maximizer of $F=E-L$, we will be done if we can show that
\begin{equation}\label{equ:lowerb}
\liminf_{k\to\infty}F(\phi_k)\ge F(\psi)
\end{equation}
for each $\psi\in\cE^1(A)$. As a first observation we note that it is enough to prove (\ref{equ:lowerb}) when $\psi$ is smooth and strictly psh. Indeed by~\cite{Dem92} we can write an arbitrary element of $\cE^1(A)$ as a decreasing sequence of smooth strictly psh weights (since $A$ is ample and functions in $\cE^1(A)$ have zero Lelong numbers) and the monotone continuity properties of $E$ and $L$ therefore show that $\sup_{\cE^1}(E-L)$ is equal to the sup of $E-L$ over all smooth strictly psh weights.

Let us now establish (\ref{equ:lowerb}) for a smooth strictly psh $\psi$. Since $F_k=D_k-L\circ\FS_k$ is maximized at $H_k$ we have in particular 
\begin{equation}\label{equ:dmin}
F_k(H_k)\ge D_k(\Hilb_k(\mu_0,\psi))-L(P_k(\psi)).
\end{equation}
Since $D_k(\Hilb_k(\mu_0,\phi_0))=0$ the first term on the right-hand side of (\ref{equ:dmin}) writes
$$
D_k(\Hilb_k(\mu_0,\psi))=\int_{t=0}^1\left(\frac{d}{dt}D_k(\Hilb_k(\mu_0,t\psi+(1-t)\phi_0))\right)dt.
$$
By~\cite{BB08} Lemma 4.1 we have
$$
\frac{d}{dt}D_k(\Hilb_k(\mu_0,t\psi+(1-t)\phi_0))=\int_X(\psi-\phi_0)\b_k(\phi_0,t\psi+(1-t)\phi_0)
$$
and the Bouche-Catlin-Tian-Zelditch theorem yields 
$$
D_k(\Hilb_k(\mu_0,\psi))\to\int_{t=0}^1\int_X(\psi-\phi_0)\MA(t\psi+(1-t)\phi_0)dt=E(\psi).
$$
(this argument is actually an easy special case of~\cite{BB08} Theorem A). The second term on the right-hand side of (\ref{equ:dmin}) satisfies
$L(P_k(\psi))\to L(\psi)$ since $P_k(\psi)\to\psi$ uniformly. It follows that
\begin{equation}\label{equ:minor}
F_k(H_k)\ge F(\psi)+o(1)
\end{equation}
(where $o(1)$ depends on $\psi$) and we will thus be done if we can show that
$$
F(\phi_k)-F_k(H_k)\ge o(1).
$$
Now we have
$$
F(\phi_k)-F_k(H_k)=\left(J_k-J\circ\FS_k\right)(H_k)\ge-\e_kJ_k(H_k)+o(1)
$$
by (\ref{equ:estim}) so it is enough to show that $J_k(H_k)$ is bounded from above. But we can apply the uniform coercivity estimate of Lemma~\ref{lem:proper} to get
$$
F_k(H_k)\le-\e J_k(H_k)+O(1)
$$
for some $\e>0$. Since the left-hand side is bounded from below in view of (\ref{equ:minor}) we are finally done.

\end{document}